\documentclass[a4paper,11pt]{amsart}
\usepackage{amssymb,mathrsfs}
\usepackage{amsthm}
\usepackage{amsmath}
\usepackage{latexsym}
\usepackage{fancyhdr}
\usepackage{ascmac}
\usepackage{color}
\usepackage[all]{xy}
\usepackage{amscd}

\theoremstyle{plain}
\newtheorem{thm}{Theorem}[section]
\newtheorem{prop}[thm]{Proposition}
\newtheorem{cor}[thm]{Corollary}
\newtheorem{lem}[thm]{Lemma}
\newtheorem{dfn}[thm]{Definition}

\newtheorem{rmk}[thm]{Remark}
\makeatletter
 
  \@addtoreset{equation}{section}
\makeatother

\newcommand{\bQ}{\overline{\mathbb{Q}}}

\newcommand{\bF}{\overline{\mathbb{F}}}

\newcommand{\C}{\mathbb{C}}

\newcommand{\Q}{\mathbb{Q}}

\newcommand{\Z}{\mathbb{Z}}

\newcommand{\F}{\mathbb{F}}
\newcommand{\GSp}{{\rm GSp}}

\newcommand{\lra}{\longrightarrow}

\newcommand{\E}{\mathcal{E}}

\renewcommand{\O}{\mathcal{O}}

\newcommand{\br}{\overline{\rho}}

\newcommand{\ds}{\displaystyle}

\newcommand{\bM}{\overline{M}}

\newcommand{\uk}{\underline{k}}
\newcommand{\w}{\omega}
\newcommand{\ochi}{\overline{\chi}}
\newcommand{\op}{\overline{\psi}}
\newcommand{\ot}{\overline{\tau}}

\newcommand{\Gp}{G_{\Q_p}}
\newcommand{\ve}{\overline{\varepsilon}}
\newcommand{\cyc}{{\rm Cyc}}

\title[A new theta cycle for $GSp_4$ and an Edixhoven type theorem]
{A new theta cycle for $GSp_4$ and an Edixhoven type theorem}
\author{Takuya Yamauchi}
\thanks{During this research the author
was partially supported by JSPS KAKENHI Grant Number (B) No.19H01778 until the end of March 2024.}

\address{Takuya Yamauchi \\ 
Mathematical Inst. Tohoku Univ.\\
 6-3,Aoba, Aramaki, Aoba-Ku, Sendai 980-8578, JAPAN}
\email{takuya.yamauchi.c3@tohoku.ac.jp}

\begin{document}

\maketitle

\begin{abstract}
In this paper, we investigate a new theta cycle for $GSp_4/\Q$ by using 
author's theta operators defined in the previous work. 
In the course of the construction, we also modify the theta operators so that they work in any characteristic, including $p = 2$, and for any weight. As an application, we discuss 
an Edixhoven type theorem for $GSp_4/\Q$. 
\end{abstract}


\section{Introduction}
Mod $p$ modular forms in a broad sense play an important role in  
the arithmetic study of Galois representations (cf. \cite{serre-c}, \cite{serre}, and \cite{Edix} 
among others). 
After Serre-Katz-Jochnowitz's works (cf. \cite{joch} and \cite{katz}), 
several people have studied a variant of the theta operator which is 
a mod $p$ analogue of the Maass-Shimura differential operator 
(or the Ramanujan-Serre differential in the case of elliptic modular forms).  
We refer the reader to \cite{yam} and \cite{Ortiz}. We also highlight \cite{EFGMM} as a particularly relevant reference on this topic, which includes a comprehensive introduction.

In this paper we recast the theta operators for $GSp_4/\Q$ constructed in \cite{yam} and 
construct a new ``big'' theta operator acting on any geometric Siegel 
modular forms over $\bF_p$ of degree 2. A novelty is that we allow any weight and any characteristic of the base field.   
Let us fix some notation to explain the main results and refer to the appropriate sections for 
details. 
Let $p$ be any prime including $2$ and $N\ge 3$ be an integer with $p\nmid N$. 
For each pair $\uk=(k_1,k_2)$ of integers with $k_1\ge k_2$ where we allow $k_2$ negative, we denote by 
$M_{\uk}(N,\bF_p)$ be the space of the geometric (Siegel) modular forms over $\bF_p$ of 
weight $\uk$ with respect to the principal congruence subgroup $K(N)\subset \GSp_4(\widehat{\Z})$ (cf. \cite[Section 2.5]{yam}). Put 
$$m_{\uk}:=\left\{\begin{array}{ll}
(0,0) & \text{if $p=2$,} \\
(p-1,p-1) & \text{if $p>2$ and $k_1-k_2\le 1$,}\\
(2p-2,2p-2)  & \text{if $p>2$ and $k_1-k_2> 1$.}
\end{array}\right.
$$
Our first main result is the following (see Theorem \ref{extension1} and Theorem \ref{theta3}):
\begin{thm}\label{main1}
There is an $\bF_p$-linear operator 
$\Theta_{\uk}:M_{\uk}(N,\bF_p)\lra M_{\uk+(2,2)+m_{\uk}}(N,\bF_p)$ satisfying the properties below:
\begin{enumerate}
\item if $f\in M_{\uk}(N,\bF_p)$ is a Hecke eigenform outside $pN$ 
$($outside $N$ if $k_2\ge 2)$, then so is 
$\Theta_{\uk}(f)$;
\item if $f\in M_{\uk}(N,\bF_p)$ is a Hecke eigen cusp form outside $pN$ and 
$\Theta_{\uk}(f)$ is not identically zero, then $\br_{\Theta_{\uk}(f),p}\simeq 
\overline{\chi}^2_p\otimes \br_{f,p}
$ for the corresponding mod $p$ Galois representations 
$($cf. \cite{wei1},\cite{wei},\cite{wei2}$)$ of $G_\Q:={\rm Gal}(\bQ/\Q)$. Here 
$\overline{\chi}_p$ stands for the mod $p$ cyclotomic character of $G_\Q$. 
\end{enumerate}
\end{thm}
A motivation to construct $\Theta_{\uk}$ is to study the filtration of a non-zero element $f\in M_{\uk}(N,\bF_p)$ 
which is defined by 
\begin{equation}\label{w(f)}
w(f):=\min\{\uk-i(p-1,p-1)\in \Z^2\ |\ f\in (H_{p-1})^i\cdot M_{\uk-i(p-1,p-1)}(N,\bF_p)\}
\end{equation}
with respect to the lexicographic order on $\Z^2$ so that the second entries are firstly compared.  
Here $H_{p-1}$ is the Hasse invariant of degree 2 and it can be regarded as 
a geometric Siegel modular form over $\bF_p$ of parallel weight $(p-1,p-1)$ with level one. 
The filtration $w(f)$ of $f$ is well-defined since the zero locus of $H_{p-1}$ is an irreducible  divisor of 
the Siegel 3-fold $S_{N,p}$ (cf. \cite[p.349, Corollary (1.5)]{Oort}). 
Since the Hasse invariant has a scalar weight, it is natural to consider the construction of a differential operator that increases the weight by a scalar amount and also interacts with Hecke eigenvalues. In the author's previous work \cite{yam}, this construction was carried out only in the case $k_1=k_2$, under the assumption $p\ge 5$. In contrast our big theta $\Theta_{\uk}$ works for any weight 
$\uk=(k_1,k_2)$ with $k_1\ge k_2$ and any prime $p$. 

For each $f\in M_{\uk}(N,\bF_p)$, the theta cycle of $f$ with respect to $\Theta=\Theta_{\uk}$ 
is defined by 
$$\cyc(f)=
\left\{
\begin{array}{ll}
(w(\Theta(f)),w(\Theta^2(f)),\ldots,w(\Theta^{\frac{p-1}{2}}(f))) & \text{if $p>2$}\\
(w(\Theta(f)))  & \text{if $p=2$.}
\end{array}\right.
$$
We remark that $\Theta^i(f)\neq 0$ for any positive integer $i$ if $f$ is not weakly $p$-singular (see Definition \ref{p-singular} and Theorem \ref{non-vanishing}).
Combining with the automorphy lifting theorems due to Gee-Geraghyty \cite{gg} which 
are extended by the author \cite{yam2} we prove the following:
\begin{thm}\label{main2}
Let $\br:G_\Q\lra \GSp_4(\bF_p)$ be a mod $p$ Galois representation satisfying 
\begin{itemize}
\item $p\ge 3$; 
\item $\br|_{G_{\Q(\zeta_p)}}$ is irreducible and ${\rm Im}(\br)$ is adequate; 
\item $\br\simeq \br_{f,p}$ for some cuspidal Hecke eigenform $f$ in $M_{\uk'}(N,\bF_p)$ 
with $\uk'=(k'_1,k'_2),\ k'_1\ge k'_2\ge 3$ and $N\ge 3$ satisfying $p\nmid N$.  
\end{itemize}
Then there exist a cuspidal Hecke eigen form $g$ in  $M_{(k_1(\br),k_2(\br))}(N,\bF_p)$ 
of the classical Serre weight $(k_1(\br),k_2(\br))$ with the integer $w(\br)$ defined in Section \ref{CSW1}  
such that $\br\simeq \ochi^{w(\br)}_p\otimes\br_{g,p}$. Further, if $g$ is not weakly $p$-singular, then 
$$\br\simeq\left\{
\begin{array}{ll}
\br_{\Theta^j(g),p} & \text{if $w(\br)\equiv 2j$ {\rm mod} $p-1$ 
with $0\le j <\frac{p-1}{2}$,} \\
\br_{\theta_3\circ \Theta^j(g),p} & \text{if $w(\br)\equiv 2j+1$ {\rm mod} $p-1$ 
with $0\le j <\frac{p-1}{2}$ and $\Theta^j(g)|_{S_{(0,0)}}\not\equiv 0$} 
\end{array}\right.
$$
where the filtration of $\Theta^j(g)$ appears in the theta cycle ${\rm Cyc}(g)$ of $g$. 
Here $\theta_3$ stands for the small theta operator defined in Section \ref{small-theta} and $S_{(0,0)}$ stands for the 
superspecial locus of $S_{N,p}$. 
\end{thm}

The classical Serre weights are closely related to an optimal weight of Siegel Hecke eigen cusp forms of level prime to $p$ 
over characteristic zero that give rise to mod $p$ Galois representations equivalent to $\br$. According to
a conjectural mod $p$ local Langlands correspondence (cf.\cite{GHS}), these weights should be interpreted
in terms of the restriction $\br|_{G_{\Q_p}}$. Our definition of the classical Serre weights follows this perspective.

The geometric modular form $g$ in the above claim, in fact, comes from the reduction of 
a geometric modular form of such a weight over a field of characteristic zero since it is obtained by an automorphy lifting 
theorem. 

If $\br|_{G_{\Q_p}}$ is irreducible, it is well-known that it has a 
potentially diagonalizable lift of Hodge-Tate weight $\{0,1,2,3\}$ but it may not be crystalline (cf. \cite[p.410, Lemma 2.1.12]{GHTS}). 
Therefore, if we do not requite the lift to be crystalline, then 
we have a similar result as recorded in Theorem \ref{main3}. 
A theorem of this kind has become well known following \cite{BGGT}, together with \cite{GHTS} and \cite{EG}.

The classical Serre weight for $\br$ is defined in Section \ref{CSW1} for $p>2$ and 
in Section \ref{CSW2} for $p=2$ according to the shape of $\br|_{G_{\Q_p}}$. More precisely, 
assume $\br|_{\Gp}$ is $\ast$-ordinary where $\ast\in\{\text{Borel,Siegel,Klingen}\}$. 
Then the classical Serre weight will be closely related to 
the Hodge-Tate weights of a potentially diagonalizable crystalline lift of $\br|_{\Gp}$ which preserves $\ast$-ordinary. In \cite{yam2}, the classical Serre weights are defined but
not explicit in some cases. Here we give a precise definition as in \cite{serre},\cite{Edix}. 

This paper is organized as follows. 
In Section \ref{theta-operators}, after preliminaries for geometric objects, we construct our theta operators (big theta and small theta) which work for any characteristic and any weights.  
A detailed study of the local behavior of the theta operator is a balk of this paper. 
A new phenomena is observed when we discuss the entire extension of the theta operator. 
The theta cycle is defined by using this operator and basic properties are discussed in 
Section \ref{the-theta-cycle}. In Section \ref{CSW1} and \ref{CSW2}, we give the definition of the classical Serre weights 
for a given $\br$. 
Finally, we give a proof of Theorem \ref{main2} in Section \ref{proof}. 

\textbf{Acknowledgment.} The author would like to thank
Professor B\"ocherer for answering several questions about $p$-singular forms.
He would also like to express special thanks to the referee for carefully reading the manuscript,
pointing out some errors in an earlier version, and providing valuable comments.
In particular, Remark \ref{weights}-(2) was brought to the author's attention by the referee.

\section{Theta operators}\label{theta-operators} 
In this section, we give a modification of the theta operators defined in Section 3 of \cite{yam}, 
so that they work in any characteristic and for any weight.
We refer to Section 2 and 3 of \cite{yam} for the notation.  

We denote by $GSp_4$ the symplectic group with respect to 
$J=\begin{pmatrix}
0_2 & s \\
-s & 0_2
\end{pmatrix},\ s=\begin{pmatrix}
0 & 1 \\
1 & 0
\end{pmatrix}$ with the similitude. 
This is a smooth group scheme over $\Z$. 

\subsection{Geometric modular forms}\label{gmm}
Let $N\ge 3$ be an integer  and  
$S_{K(N)}$ be the Siegel modular threefold over $\Z[1/N]$ with respect to the principal  congruence subgroup $K(N)\subset \GSp_4(\widehat{\Z})$. 
Let $f:\mathcal{A}\lra S_{N}$ be the universal abelian surface and 
we define the Hodge bundle $\E:=f_\ast \Omega^1_{\mathcal{A}/S_{N}}$ which is 
a locally free sheave on $S_{N}$ of rank 2. Put $\w:=\det(\E)$. 
For each pair $\uk=(k_1,k_2)$ of integers with $k_1\ge k_2$ where we allow 
$k_2$ non-positive. 
Put $\E_{\uk}:={\rm Sym}^{k_1-k_2}\E\otimes_{\O_{S_{N,p}}}\w^{k_2}$ 
(this is denoted by $\w_{\uk}$ in \cite{yam}) and for any $\Z[1/N]$-algebra $R$, we define
$M_{\uk}(N,R):=H^0(S_{N}\otimes_{\Z[1/N]}R,\E_{\uk}\otimes_{\Z[1/N]}R)$. 
Each element of  $M_{\uk}(N,R)$ is said to be 
a geometric (Siegel) modular form over $R$ of weight $\uk$ with respect to $K(N)$. 
\begin{rmk}\label{weights}Keep the notation being as above. 
\begin{enumerate} 
\item It is well-known that if $k_2<0$, then $M_{\uk}(N,\C)=0$ {\rm(}cf. \cite{F}{\rm)}.  
\item For any $p$, the space $M_{\uk}(N,\bF_p)$ could be non-zero even if $k_2$ is negative. 
In fact, the Verschiebung map $V$ defined in \cite[p.26, (3.37)]{yam} gives an non-zero element in 
 $M_{(p,-1)}(N,\bF_p)$. Further, for each positive integer $n\ge 1$, a certain component of $V^{\otimes n}$ gives  
 a non-zero element in $M_{(np,-n)}(N,\bF_p)$. For a more conceptual perspective, see \cite[Theorem 5.1.1]{GK}.     
\end{enumerate}
\end{rmk}

\subsection{Gauss-Manin connection}
Let $\mathbb{H}^1_{{\rm dR}}(\mathcal{A}/S_{N})$ be the algebraic de Rham cohomology sheaf on $S_{N}$. 
Let $\nabla:\mathbb{H}^1_{{\rm dR}}(\mathcal{A}/S_{N})
\lra \mathbb{H}^1_{{\rm dR}}(\mathcal{A}/S_{N})\otimes_{\O_{S_{N}}}
\Omega^1_{S_{N}}$ be 
the Gauss-Manin connection. It yields the Kodaira-Spencer isomorphism 
$${\rm KS}:{\rm Sym}^2\E\stackrel{\sim}{\lra}\Omega^1_{S_{N}},\ 
\omega_1\otimes \omega_2\mapsto \langle \omega_1,\nabla\omega_2 \rangle_{{\rm dR}}$$
where $\langle\ast,\ast \rangle_{{\rm dR}}$ stands for the alternating pairing on  
$\mathbb{H}^1_{{\rm dR}}(\mathcal{A}/S_{N})$. 
We remark that the formation of KS is compatible with the base change to any 
$\Z[1/N]$-algebra.
\subsection{A non-canonical projection $p_1$}\label{non-can}
Fix a local basis $e_1,e_2$ of $\E$. Put $u_{i}=e^{i}_1e^{2-i}_2,i=0,1,2$ 
with the convention $e^0_1=e^0_2:=1$ 
 which make up a local basis of 
${\rm Sym}^2\E$. To avoid confusion, we prepare additional symbols  $v_{i}=e^{i}_1e^{2-i}_2,i=0,1,2$ which play the same role. We introduce a non-canonical projection 
\begin{equation}\label{proj}
p_1:{\rm Sym}^2\E\otimes_{\O_{S_{N}}}{\rm Sym}^2\E\lra  
\w^{\otimes 2}
\end{equation}
as ${\rm Aut}_{\O_{S_{N}}}(\E)$-modules.  
This can be given explicitly as follows. 
For any local section $x=\ds\sum_{0\le i,j\le 2}a_{ij}u_i\otimes v_j$ of 
${\rm Sym}^2\E\otimes_{\O_{S_{N}}}{\rm Sym}^2\E$ we define 
\begin{equation}\label{decom-explicit}
p_1(x)=(2a_{20}-a_{11}+2a_{02})(e_1\wedge e_2)^2.  
\end{equation}   
By direct computation, for any local section $\gamma$ of ${\rm Aut}_{\O_{S_{N}}}(\E)$ 
we see $p_1(\gamma x)=\det(\gamma)^2 p_1(x)$.
We also remark that the formation of $p_1$ is compatible with the base change to any 
$\Z[1/N]$-algebra. 
\subsection{The theta operator for the projection $p_1$}\label{tp1}
Let $p$ be any prime including 2  
and $N\ge 3$ be a positive integer with $p\nmid N$. 
Let $S_{N,p}$ be a connected component of the special fiber $S_{K(N)}\otimes_{\Z[1/N]} \bF_p$ at $p$.  
We work locally on the ordinary locus of $S_{N,p}$. 
By abusing notation we use the same symbols $\mathcal{A},\E,\w,\E_{\uk}$ to denote 
their base change to $\bF_p$. 
Then we have a unit root decomposition 
$\mathbb{V}:=\mathbb{H}^1_{{\rm dR}}(\mathcal{A}/S_{N,p})=U\oplus \E$ 
where $U={\rm Im}(F^\ast:\mathbb{V}^{(p)}\lra \mathbb{V})$ is the image of the pullback of the relative Frobenius map $F$ on $\mathcal{A}/S_{N,p}$ (see Section 3.3.1 of \cite{yam}). 
For a pair $\uk=(k_1,k_2)$ of integers  with $k_1\ge k_2$, we define 
$\mathbb{V}_{\uk}:={\rm Sym}^{k_1-k_2}\mathbb{V}\otimes_{\O_{S_{N,p}}} 
(\wedge^2\mathbb{V})^{\otimes k_2}$ where 
$(\wedge^2\mathbb{V})^{\otimes k_2}:={\mathcal{H}om}_{\O_{S_{N,p}}}((\wedge^2\mathbb{V})^{\otimes |k_2|},\O_{S_{N,p}})$ if $k_2$ is negative. 
The unit-root decomposition yields a decomposition 
$\mathbb{V}_{\uk}=\E_{\uk}\oplus R_{\uk}(U)$ and we define 
a natural inclusion 
$\iota_{\uk}:\E_{\uk}\hookrightarrow \mathbb{V}_{\uk}$ with respect to this decomposition.
We are now ready to define our ``pre''-theta operator $\widetilde{\Theta}_{\uk}$ 
and two maps $\Psi^{(i)}_{\uk},\ i=1,2$ by the following commutative diagram:
\begin{equation}\label{scalar-case}
\xymatrix{
\E_{\uk}  \ar[r]^{\iota_{\uk}\hspace{3mm}}  \ar@/_45pt/[dddrrr]_{\widetilde{\Theta}_{\uk}} 
 \ar@/_35pt/[ddrrr]_{\Psi^{(2)}_{\uk}} 
 \ar@/_25pt/[rrr]_{\Psi^{(1)}_{\uk}} & 
\mathbb{V}_{\uk}\ar[r]^{{\rm KS}^{-1}\circ \nabla\hspace{13mm}}&  
\mathbb{V}_{\uk}\otimes_{\O_{S_{N,p}}}{\rm Sym}^2\E 
\ar[r]^{{\rm id}\otimes\iota_{(2,0)}} & \mathbb{V}_{\uk}\otimes_{\O_{S_{N,p}}}
\mathbb{V}_{(2,0)}
\ar[d]^{{\rm KS}^{-1}\circ\nabla}   \\
& & &\ar[d]_{\mod R_{\uk}(U)\otimes_{\O_{S_{N,p}}}R_{(2,0)}(U)}   \mathbb{V}_{\uk}\otimes_{\O_{S_{N,p}}}
\mathbb{V}_{(2,0)}
\otimes_{\O_{S_{N,p}}}{\rm Sym}^2\E  \\
& & &\ar[d]^{{\rm id}\otimes p_1}  \E_{\uk}\otimes_{\O_{S_{N,p}}}({\rm Sym}^2\E
\otimes_{\O_{S_{N,p}}}{\rm Sym}^2\E)\\ 
& & &   \E_{\uk}\otimes_{\O_{S_{N,p}}}\w^2=\E_{\uk+(2,2)}.
} 
\end{equation}  
We shall compute the local behavior of $\widetilde{\Theta}_{\uk}$. 
Fix a local basis $e_1,e_2$ of $\E$. Put $u_{i}=e^{i}_1e^{2-i}_2,i=0,1,2$ 
(resp. $v_{i}=e^{i}_1e^{2-i}_2,i=0,1,2$)
with the convention $e^0_1=e^0_2:=1$ 
 which make up a local basis of the first (resp. second) ${\rm Sym}^2\E$ in the 
 target of $\Psi^{(2)}_{\uk}$.  
We also consider $\delta_n:=e^{k_1-k_2-n}_1e^n_2(e_1\wedge e_2)^{k_2}$ for 
$0\le n\le r:=k_1-k_2$. Then $\{\delta_n\}^r_{n=0}$ makes up a local basis of $\E_{\uk}$. 
For $1\le i\le j\le 2$, let $\nabla_{ij}=\nabla(D_{ij})$ with 
$D_{ij}:=\langle e_i,\nabla(e_j) \rangle_{{\rm dR}}$. 
As in (3.21) of Section 3.3 of \cite{yam}, we consider 
\begin{equation}\label{AB}
\begin{array}{c}
F^\ast((\nabla_{11}(e_1))^{(p)},(\nabla_{22}(e_2))^{(p)})=(e_1,e_2,\nabla_{11}(e_1),\nabla_{22}(e_2))
\left(
\begin{array}{c}
B\\
A
\end{array}
\right),\\
 B=\left(
\begin{array}{cc}
b_{11} & b_{12} \\
b_{21} & b_{22} 
\end{array}
\right),\ 
A=\left(
\begin{array}{cc}
a_{11} & a_{12} \\
a_{21} & a_{22} 
\end{array}
\right)\in M_2(\O_{S_{N,p}})
\end{array}
\end{equation}
and put 
\begin{equation}\label{C}
C=BA^{-1}=\left(
\begin{array}{cc}
c_{11} & c_{12} \\
c_{21} & c_{22} 
\end{array}
\right). 
\end{equation}
Notice that $A$ is nothing but the Hasse matrix of the universal abelian surface 
and it is nowhere vanishing on the ordinary locus of $S_{N,p}$.  
We first compute the local behavior of $\Psi^{(1)}_{\uk}$. 
\begin{prop}\label{phi1}Recall $r=k_1-k_2$. Let $F=\ds\sum_{n=0}^r F_n\delta_n$ be 
a local section of $\E_{\uk}$.  It holds that 
$$\Psi^{(1)}_{\uk}(F)\equiv 
\Bigg(\sum_{n=0}^r b^{(n)}_{2}\delta_n\otimes u_2  \Bigg)+
\Bigg(\sum_{n=0}^r b^{(n)}_{1}\delta_n\otimes u_1  \Bigg)+
\Bigg(\sum_{n=0}^r b^{(n)}_{0}\delta_n\otimes u_0  \Bigg)
$$
modulo $\mathbb{V}_{\uk}\otimes_{\O_{S_{N,p}}}R_{(2,0)}(U)$ where 
\begin{eqnarray}
b^{(n)}_{2}&=&\nabla_{11}(F_n)-(k_1-n)F_n c_{11}-(r+1-n)F_{n-1}c_{12}\nonumber \\
b^{(n)}_{1}&=&\nabla_{12}(F_n)-(n+1)F_{n+1} c_{11}-F_n\{(k_1-n)c_{21}+(k_2+n)c_{12}\}
-(r+1-n)F_{n-1}c_{22} \nonumber\\
b^{(n)}_{0}&=&\nabla_{22}(F_n)-(k_2+n)F_n c_{22}-(n+1)F_{n+1}c_{21}.\nonumber
\end{eqnarray}
\end{prop}  
\begin{proof}By definition, we have 
$$\Psi^{(1)}_{\uk}(F)= 
\Bigg(\sum_{n=0}^r (\nabla_{11}(F_n)\delta_n+F_n\nabla_{11}(\delta_n))\otimes u_2  \Bigg)+
\Bigg(\sum_{n=0}^r (\nabla_{12}(F_n)\delta_n+F_n\nabla_{12}(\delta_n))\otimes u_1  \Bigg)$$
$$+\Bigg(\sum_{n=0}^r (\nabla_{22}(F_n)\delta_n+F_n\nabla_{22}(\delta_n))\otimes u_0  \Bigg).
$$
The claim follows from the formula (3.35) of \cite{yam}. 
\end{proof}
Before going further, we need the following lemmas 
for the local sections appearing in (\ref{AB}) and (\ref{C}). 
\begin{lem}\label{detA}
For the local sections appearing in $($\ref{AB}$)$ and $($\ref{C}$)$, it 
holds that 
$$\nabla_{11}(\det (A))=-c_{11}\cdot\det(A),\ \nabla_{12}(\det (A))=
-(c_{12}+c_{21})\cdot\det(A),\ \nabla_{22}(\det (A))=-c_{22}\cdot\det(A)$$
\end{lem}
\begin{proof}
The claim follows from Proposition 3.4-(1) of \cite{yam}. 
\end{proof}

\begin{lem}\label{cc}
Suppose that all $\nabla_{ij}$'s $(1\le i\le j\le 2)$ commute with each other. 
Then it holds that 
$$\begin{array}{lll} 
\nabla_{11}(c_{11})=c^2_{11}     & \nabla_{12}(c_{11})=(c_{12}+c_{21})c_{11} & \nabla_{22}(c_{11})=c_{12}c_{21}\\
\nabla_{11}(c_{12})=c_{11}c_{12} & \nabla_{12}(c_{12})=c_{11}c_{22}+c^2_{12} & \nabla_{22}(c_{12})=c_{22}c_{12}\\
\nabla_{11}(c_{21})=c_{11}c_{21} & \nabla_{12}(c_{21})=c_{11}c_{22}+c^2_{21} & \nabla_{22}(c_{21})=c_{22}c_{21}\\
\nabla_{11}(c_{22})=c_{12}c_{21} & \nabla_{12}(c_{22})=(c_{12}+c_{21})c_{22} & \nabla_{22}(c_{22})=c^2_{22}.
\end{array}
$$
\end{lem}
\begin{proof}
For each $1\le i\le \le 2$, let $n_{ij}=c_{ij}\cdot\det(A)$. For each $1\le k\le l\le 2$, we have  
$$\nabla_{kl}(c_{ij})=\frac{\nabla_{kl}(n_{ij})}{\det(A)}-\frac{\nabla_{kl}(\det(A))}{\det(A)}c_{ij}.$$
The first term will be computed by Proposition 3.4-(1),(2) of \cite{yam} with the 
assumption on $\nabla_{kl}$'s and the second term is done by Lemma \ref{detA}. 
\end{proof}
The commutativity of the differential operators yields an important property as below. 
\begin{lem}\label{c21}
Keep the notation in $($\ref{AB}$)$ and $($\ref{C}$)$. 
Suppose that all $\nabla_{ij}$'s $(1\le i\le j\le 2)$ commute with each other. 
Then it holds $c_{12}=c_{21}$. 
\end{lem}
\begin{proof}For $i=0,1$, let $\widetilde{\eta}_i$ be  the local section of 
$R^1 f_\ast \O_{\mathcal{A}}$ corresponding to $e_i$ under the Serre duality. 
By using Lemma 3.2 of \cite{yam} and the Leibniz rule, one can check 
$$\nabla_{ij}\langle \nabla_{11}(e_1), \nabla_{22}(e_2)\rangle_{{\rm dR}}=0,\ 1\le i\le j\le 2.$$
It follows from this that $\langle \nabla_{11}e_1, \nabla_{22}e_2\rangle_{{\rm dR}}=0$.  
By (3.28) of \cite{yam}, we have 
$$\nabla_{11}e_1=\eta_1-(c_{11}e_1+c_{12}e_2),\ \nabla_{22}e_2=\eta_2-(c_{21}e_1+c_{22}e_2)$$
where $\eta_i$ ($i=0,1$) is a lift of $\widetilde{\eta}_i$ to $\mathbb{V}$. 
It yields 
$$0=\langle \nabla_{11}(e_1), \nabla_{22}(e_2)\rangle_{{\rm dR}}=c_{12}-c_{21}$$
which gives us the claim.
\end{proof}
Let $F=\ds\sum_{n=0}^r F_n\delta_n$ be 
a local section of $\E_{\uk}$.  
The computation of $\widetilde{\Theta}_{\uk}(F)=
({\rm id}\otimes p_1)\circ \Psi^{(2)}_{\uk}(F)$ goes as follows:
\begin{enumerate}
\item first we compute $\nabla_{kl}(b^{(n)}_{i})$ and 
$\nabla_{kl}(\delta_n)$ for $1\le k\le l\le 2,\ i=0,1,2$, and 
$0\le n\le r=k_1-k_2$; 
\item next we collect the coefficients of  $u_2\otimes v_0,\ u_1\otimes v_1, u_0\otimes v_2$ 
to compute ${\rm id}\otimes p_1$. 
\end{enumerate}
The resulting terms involve some of $\nabla_{kl}(c_{ij}),c_{ij}$ and we use Lemma \ref{cc} 
and Lemma \ref{c21} 
to simplify the equation. 
Summing up, we have the following explicit form of the local behavior of $\widetilde{\Theta}_{\uk}$. 
\begin{prop}\label{local}Recall $r=k_1-k_2$. Suppose that all $\nabla_{ij}$'s $(1\le i\le j\le 2)$ commute with each other. Let $F=\ds\sum_{n=0}^r F_n\delta_n$ be 
a local section of $\E_{\uk}$ on each open subscheme of the ordinary locus of 
$S_{N,p}$. Let $\widetilde{\Theta}_{\uk}(F)=
\ds\sum_{n=0}^r A_n \delta_n (e_1\wedge e_2)^2$.
Then it holds 
\begin{eqnarray}
A_n&=&\det\begin{pmatrix}
2\nabla_{11} & \nabla_{12}  \\
\nabla_{12} & 2\nabla_{22}
\end{pmatrix}(F_n)+2 \{(k_1(-1 + 2 k_2) +( k_1  - k_2  - n)n\}\det(C) F_n \nonumber\\
&& -2 (-1 + 2 k_2 + 2 n) c_{22}\nabla_{11}(F_{n})+
2 (-1 + 2 k_1+2k_2) c_{12}\nabla_{12}(F_{n}) +2 (1 - 2 k_1 + 2 n) c_{11}\nabla_{22}(F_{n}) \nonumber\\
&&-\{( k_1 - k_2-1) (k_1 - k_2) - 6 (k_1 - k_2 - n) n\}c^2_{12}F_n \nonumber \\
&& +2 (1 + k_1 - k_2 - n) c_{22}\nabla_{12}(F_{n-1})-
4 (1 + k_1 - k_2 - n) c_{12} \nabla_{22}(F_{n-1}) \nonumber\\
&&-4(n+1)c_{12}\nabla_{11}(F_{n+1})+2(n+1)c_{11}\nabla_{12}(F_{n+1}) \nonumber\\
&& +2(1 + k_1 - k_2 - n)(-1 +  k_2 -k_1+ 2 n)c_{12}c_{22}F_{n-1}\nonumber\\
&& -2(1+n)(-k_1+k_2+1 + 2n)c_{11}c_{12}F_{n+1}\nonumber\\
&& -(1 + k_1 - k_2 - n) (2 + k_1 - k_2 - n) c^2_{22}F_{n-2}-(1 + n) (2 + n)c^2_{11}F_{n+2}\nonumber
\end{eqnarray} 
where the terms involving $\nabla_{kl}(F_{n+i})$ or $F_{n+i}$ with 
$i\in\{0,\pm 1,\pm 2\}$ and $1\le k\le l\le 2$ are ignored if $n+i$ is out of the range for the index. 
\end{prop}
\begin{rmk}The symbolic computation in Proposition \ref{local} is done by using 
Wolfram Mathematica 12.1.
\end{rmk}
We write down the formula in Proposition \ref{local} in the case when $r=0$ or $r=1$ 
respectively. 
It turns out later that these cases have a special feature among others.  
\begin{cor}\label{hol1}Keep the notation and the assumption being as above. 
Suppose $k:=k_1=k_2$ so that $r=0$. Let $F=F_0\delta_0$ be a local section of 
$\E_{(k,k)}=\w^{\otimes k}$ and put $\widetilde{\Theta}_{\uk}(F)=A_0
(\delta_0\otimes (e_1\wedge e_2)^2)$. Then it holds that  
$$A_0=\det\begin{pmatrix}
2\nabla_{11} & \nabla_{12}  \\
\nabla_{12} & 2\nabla_{22}
\end{pmatrix}(F_0)+2k(2k-1)\det(C) F_0$$
$$-2(2k-1)\{c_{22}\nabla_{11}(F_{0})-c_{12}\nabla_{12}(F_{0})+c_{11}\nabla_{22}(F_{0})\}.$$
\end{cor}
\begin{cor}\label{hol2}Keep the notation and the assumption being as above. 
Suppose $(k_1,k_2)=(k+1,k)$ so that $r=1$. Let $F=F_0\delta_0+F_1\delta_1$ be a local section of 
$\E_{(k+1,k)}$ and put $\widetilde{\Theta}_{\uk}(F)=A_0
(\delta_0\otimes (e_1\wedge e_2)^2)+A_1
(\delta_1\otimes (e_1\wedge e_2)^2)$. Then it holds that 
$$A_0=\det\begin{pmatrix}
2\nabla_{11} & \nabla_{12}  \\
\nabla_{12} & 2\nabla_{22}
\end{pmatrix}(F_0)+2(k+1) (2k-1)\det(C) F_0 $$
$$ -2 (2 k-1) c_{22}\nabla_{11}(F_{0})+
4k c_{12}\nabla_{12}(F_{0}) 
-2 (2 k+1) c_{11}\nabla_{22}(F_{0})
-4c_{12}\nabla_{11}(F_{1})+2c_{11}\nabla_{12}(F_{1}) $$
and 
$$A_1=\det\begin{pmatrix}
2\nabla_{11} & \nabla_{12}  \\
\nabla_{12} & 2\nabla_{22}
\end{pmatrix}(F_1)+2 (k+1)(2k-1)\det(C) F_1 $$
$$-2 (1 + 2 k  ) c_{22}\nabla_{11}(F_{1})+
2 (1 +4k) c_{12}\nabla_{12}(F_{1}) +2(1-2 k) c_{11}\nabla_{22}(F_{1})$$
$$ +2  c_{22}\nabla_{12}(F_{0})-
4  c_{12} \nabla_{22}(F_{0}).$$
\end{cor}
Let $H_{p-1}=\det(A)$ be the Hasse invariant which can be regarded as a non-zero element of 
$H^0(S_{N,p},\w^{\otimes(p-1)})$. 
\begin{dfn}\label{theta-op}For each pair $\uk=(k_1,k_2)$ with $k_1\ge k_2$, 
define the theta operator for the weight $\uk$ by 
$$\Theta_{\uk}:=\left\{\begin{array}{ll}
 \widetilde{\Theta}_{\uk} &\text{if $p=2$} \\
H_{p-1}\cdot \widetilde{\Theta}_{\uk} & \text{if $p>2$ and $k_1-k_2\le 1$}\\
H^2_{p-1}\cdot \widetilde{\Theta}_{\uk} & \text{if $p>2$ and $k_1-k_2> 1$.}
\end{array}\right.
$$
\end{dfn}
Put 
$$m_{\uk}:=\left\{\begin{array}{ll}
(0,0) & \text{if $p=2$} \\
(p-1,p-1) & \text{if $p>2$ and $k_1-k_2\le 1$}\\
(2p-2,2p-2)  & \text{if $p>2$ and $k_1-k_2> 1$}
\end{array}\right.
$$
and $M_{\uk}(N)=M_{\uk}(N,\bF_p)$ for simplicity. 
\begin{thm}\label{extension1}Keep the notation being as above. The theta operator $\Theta_{\uk}$ 
is holomorphically extended to the whole space of $S_{N,p}$. Further, it holds that 
for each $F\in M_{\uk}(N)$, $\Theta_{\uk}(F)$ is an element of $M_{\uk+(2,2)+m_{\uk}}(N)$. 
Further it satisfies the following properties:
\begin{enumerate}
\item if $F$ is a cusp form, then so is  $\Theta_{\uk}(F)$;
\item if $F$ is a Hecke eigenform outside $pN$ $($outside $N$ if $k_2\ge 2)$, then so is 
$\Theta_{\uk}(F)$. 
In this case, if  $F$ is a cusp form and $\Theta_{\uk}(F)$ is non-zero, then 
$$\br_{\Theta_{\uk}(F),p}\simeq \overline{\chi}^2_p\otimes \br_{F,p}$$
for the corresponding  mod $p$ Galois representations of $G_\Q={\rm Gal}(\bQ/\Q)$. 
Here $\overline{\chi}_p$ stands for the mod $p$ cyclotomic character.
\end{enumerate}
\end{thm}
\begin{proof}The coefficient $A_n$ in Proposition \ref{local} may have poles 
along the divisor of $H_{p-1}=\det(A)$ except for $p=2$. The possible poles come from $c_{ij},\ 1\le i\le j\le 2$ and 
$\det(A)c_{ij}$ is holomorphic by Lemma \ref{detA}. Since $A_n$ contains at most quadratic 
monomials in $c_{ij}$'s. The holormophic extension follows from this when $r=k_1-k_2>1$. 
In the case when $r=0$ or $r=1$, the similar claim follows from Corollary \ref{hol1} and Corollary \ref{hol2} since $A_n$ contains only linear terms in $c_{ij}$'s. 

When $p=2$, we see easily by Proposition \ref{local} that 
$\Theta_{\uk}=\widetilde{\Theta}_{\uk}$ is already holomorphically extended to the whole space. 

The claim (1) follows from $q$-expansion principle and Proposition 3.10 of \cite{yam} 
yields the claim (2). 
\end{proof}
\begin{rmk}When $p>3$, it is known, by Theorem 3.4.1 of \cite{EFGMM}, that 
$H^2_{p-1}\widetilde{\Theta}_{\uk}$ is extended to the whole space $S_{N,p}$ without any explicit computation. 
Their method is conceptual and works for many interesting cases. 
However, it seems difficult to check the extension of our theta operator $\Theta_{\uk}$ 
when $k_1-k_2\in \{0,1\}$ or $p$ is any small prime.  
\end{rmk}

\subsubsection{The $($small$)$ theta operator $\theta_3=\theta^{\uk}_3$: revisited}\label{small-theta}
Recall  the contents in Section 3.3.1 and Section 3.3.2 of \cite{yam}. The author constructed 
three (small) theta operators. Among all, $\theta_3=\theta^{\uk}_3$ can be constructed 
for any $(k_1,k_2)$ with $k_1\ge k_2$ and any prime $p$. 
When $k_1-k_2\ge 2$, we may apply the explicit projection (7.1) in Appendix of \cite{yam} 
while we use $\widetilde{\theta}$ in Section 3.3.1 therein when $k_1=k_2$ and 
(3.32), (3.33) in Section 3.3.2 therein when $k_1-k_2=1$.   
Notice that the construction works for any $p$. 
Therefore, we have obtained the following:
\begin{thm}\label{theta3}
There is an $\bF_p$-linear map 
$\theta^{\uk}_3:M_{\uk}(N)\lra M_{\uk+(p+1,p-1)}(N)$ satisfying the properties below:
\begin{enumerate}
\item if $f\in M_{\uk}(N)$ is a Hecke eigenform outside $pN$ 
$($outside $N$ if $k_2\ge 2)$, then so is 
$\theta^{\uk}_3(f)$;
\item if $f\in M_{\uk}(N)$ is a Hecke eigen cusp form outside $pN$ and 
$\theta^{\uk}_3(f)$ is not identically zero, then $\br_{\theta^{\uk}_3(f),p}\simeq 
\overline{\chi}_p\otimes \br_{f,p}
$ for the corresponding mod $p$ Galois representations of $G_\Q$. 
\end{enumerate}
\end{thm} 
\begin{proof}The argument in the proof of Proposition 3.13 of \cite{yam} works also for 
this setting. 
\end{proof}

Let $S_{(0,0)}$ be the superspecial locus of $S_{N,p}$. 
\begin{thm}\label{nonvanishing-theta3}Put $r=k_1-k_2$. 
Let $f\in M_{\uk}(N)$ be a non-zero element satisfying $f|_{S_{(0,0)}}\not\equiv 0$. 
Then $\theta^{\uk}_3(f)$ is not identically zero. 
\end{thm}
\begin{proof}
We assume $r=k_1-k_2\ge 2$. 
By definition, the coefficient of  $\theta^{\uk}_3(f)$ in the basis $f^{(0)}_{r+2}$ in 
the notation of (7.1) in Appendix of \cite{yam} is 
nothing but 
$$b^{(r)}_2=\nabla_{11}(F_r)-k_2 F_r c_{11}-F_{r-1}c_{12}$$
in the notation of Proposition \ref{phi1}. 
Applying the action of ${\rm GL}_2(\bF_p)$ on $\E_{\uk}$ if necessary, 
we may assume $F_{r-1}$ is non-zero at some point of  $S_{(0,0)}$ by assumption. 
As in the proof of Theorem 4.7 of \cite{yam}, by using the local deformation at some point, 
we see that $F_{r-1}c_{12}$ is non-zero and $b^{(r)}_2$ as well.

When $k_1-k_2\in\{0,1\}$, we apply a similar argument to $\theta$ 
(see Proposition 3.5 of \cite{yam}) and (3.33) of loc.cit..
\end{proof}

\section{The theta cycle}\label{the-theta-cycle}
In this section we study the theta cycle defined by $\Theta_{\uk}$. 
\subsection{Non-vanishing results}
Let $r=k_1-k_2$. 
For each $F=\ds\sum_{n=0}^{r}F_n\delta_n\in M_{\uk}(N)$ we denote by $$F(q_N):=
\sum_{n=0}^{r}\sum_{T\in {\rm Sym^2}(\Z)}A_{F_n}(T)q^T_N\delta_n=
\sum_{n=0}^{r}\sum_{T\in {\rm Sym^2}(\Z)_{\ge 0}}A_{F_n}(T)q^T_N\delta_n=
\sum_{T\in {\rm Sym^2}(\Z)_{\ge 0}}\Big(\sum_{n=0}^{r}A_{F_n}(T)\delta_n\Big)q^T_N$$ 
the $q$-expansion of $F$ at the Mumford's semi-abelian scheme over 
$\bF_p[[q^{1/N}_{11},q^{1/N}_{12},q^{1/N}_{22}]]$ (cf. Section 2.5 of \cite{yam}). 
Here ${\rm Sym^2}(\Z)=
\Bigg\{
\begin{pmatrix}
2a & b \\
b & 2c
\end{pmatrix}\ \Bigg|\ a,b,c\in \Z  \Bigg\}$ and 
${\rm Sym^2}(\Z)_{\ge 0}=\{T\in{\rm Sym^2}(\Z)\ |\ T\ge 0  \}$. 
For each $T=
\begin{pmatrix}
2a & b \\
b & 2c
\end{pmatrix}$, we write 
$q^T_N=q^{a/N}_{11}q^{b/N}_{12}q^{c/N}_{22}$. Since $H^M_{p-1}F$ is liftable to 
a (holomorphic) Siegel modular form of characteristic zero for a sufficiently large $M$ and $H_{p-1}(q_N)=1$, 
the coefficient $A_F(T):=\ds\sum_{n=0}^{r}A_{F_n}(T)\delta_n$ vanishes unless $T\in {\rm Sym^2}(\Z)_{\ge 0}$. 
\begin{dfn}\label{p-singular}
An element $F\in M_{\uk}(N)$ is said to be 
weakly $p$-singular if $A_F(T)=0$ for any 
$T\in {\rm Sym^2}(\Z)_{\ge 0}$ with $p\nmid \det(T)$. 
\end{dfn}
This notation is slightly different from $p$-singular forms in \cite{BK}. 
The author expects that Hecke eigen weakly $p$-singular forms can be characterized by means of the image of the corresponding mod $p$ Galois representations. 

\begin{thm}\label{non-vanishing}Let $F=\ds\sum_{i=0}^r F_n\delta_n\in M_{\uk}(N)$. 
It holds:
\begin{enumerate}
\item if $F$ is not weakly $p$-singular,  then 
$\Theta_{\uk}(F)$ is neither  zero nor weakly $p$-singular;
\item if $p>2$, $(k_1,k_2)=(k,k)$, and 
$F|_{S_{(0,0)}}\not\equiv 0$, then $\Theta_{\uk}(F)|_{S_{(0,0)}}$ is not identically zero unless $p|k(2k-1)$ and 
;
\item if $p>2$, $(k_1,k_2)=(k+1,k)$, and 
$F|_{S_{(0,0)}}\not\equiv 0$, then $\Theta_{\uk}(F)|_{S_{(0,0)}}$ is not identically zero  unless $p|(k+1)(2k-1)$;
\item if $p>2$ and $k_1-k_2>1$ and $F|_{S_{(0,0)}}\not\equiv 0$, then  $\Theta_{\uk}(F)|_{S_{(0,0)}}$ is not identically zero.  
\end{enumerate}
\end{thm}
\begin{proof}
Let $F(q_N)=\ds\sum_{T\in {\rm Sym^2}(\Z)_{\ge 0}}A_F(T)q^T_N$ be the $q$-expansion of $F$. 
Since $H_{p-1}(q_N)=1$, $c_{ij}(q_N)=0$ by Lemma \ref{detA}. 
By Proposition \ref{local} we see that 
$$
\Theta_{\uk}(F)(q_N)=\sum_{n=0}^r\det\begin{pmatrix}
2q_{11}\frac{d}{dq_{11}} & q_{12}\frac{d}{dq_{12}}  \\
q_{12}\frac{d}{dq_{12}} & 2q_{22}\frac{d}{dq_{22}}
\end{pmatrix}(F_n)\delta_n=\frac{1}{N^2}\sum_{T\in {\rm Sym^2}(\Z)_{\ge 0}\atop p\nmid \det(T)}\det(T)A_F(T)q^T_N.
$$
The first claim follows from this formula. 

The second and the third claims follow from the argument in the proof of 
Theorem 4.7-(2) in \cite{yam} and Corollary \ref{hol1} and \ref{hol2}. 

For the last claim, 
by assumption, we may assume $\alpha:=F_2(X)\in \bF_p$ is non-zero for some $X\in S_{(0,0)}$ 
by using the action of ${\rm GL}_2(\bF_p)$ on $\E_{\uk}$ if necessary. 
Let $t_{11},t_{12},t_{22}$ be local parameters of $S_{N,p}$ at $X$. 
Let $R=\bF_p[[t_{11},t_{12},t_{22}]]$ with the maximal ideal $m_R$ and $I_X:=(t_{11}t_{22}-t^2_{12})R$. 
Let $A_n(X)\in R$ be the local expansion at $X$ where $A_n$ is the coefficient in 
Proposition \ref{local}. 
By the argument in the proof of 
Theorem 4.7-(2) in \cite{yam} again and by using $c_{ij}=(t_{11}t_{22}-t^2_{12})^{-1}t_{ij}$ for $1\le i\le j\le 2$ with 
$H^2_{p-1}(X)=t_{11}t_{22}-t^2_{12}$, 
\begin{equation}\label{local-exp}H^2_{p-1}(X)A_0(X)\equiv \beta_1 t^2_{12}+\beta_2 t_{12}t_{22}+\beta_3t_{12}t_{11}+
\beta_4t^2_{22}-2\alpha t^2_{11}\ {\rm mod}\ m^3_R
\end{equation}
for some $\beta_i\in \bF_p$ ($1\le i\le 4$). 
The right hand side is clearly non-zero (because of the non-zero term $2\alpha t^2_{11}$ and $p>2$) and it yields the claim. 
\end{proof}

\subsection{Filtration}\label{filtration}
We follow the contents in Section 6 of \cite{yam}. 
For each $f\in M_{\uk}(N)$ we recall the filtration $w(f)$ defined by (\ref{w(f)}).

For a multiple $M$ of $N$, let $\mathbb{T}_{M}$ be the (abstract) Hecke ring over $\Z$ outside $M$ acting on 
$M_{\uk}(N)$. 
Put 
$$\mathbb{T}=\left\{
\begin{array}{cc}
\mathbb{T}_{N}  & \text{if $k_2\ge 2$}\\
\mathbb{T}_{Np}  & \text{if $k_2< 2$}.
\end{array}\right.
$$
Then one can define the usual Hecke action of $\mathbb{T}$ on $M_{\uk}(N)$. 
The assumption $k_2\ge 2$ is necessary to guarantee that 
the formal Hecke action is defined over $\Z$ (in general 
the factor $p^{k_2-2}$ appears in the formula, cf. (2.10) of \cite{yam}).

We set the following convention for Hecke eigen forms $f_1\in M_{\uk}(N)$ and $f_2\in M_{\uk'}(N)$ where 
two weights $\uk$ and $\uk'$ are allowed to be different:  
\begin{equation}\label{conv1}
w(f_1)=w(f_2) \mbox{ if $p>2$ and the Hecke eigensystems of $f_1$ and $f_2$ for $\mathbb{T}$ are equal.}
\end{equation}
Recall 
$$m_{\uk}:=\left\{\begin{array}{ll}
(0,0) & \text{if $p=2$} \\
(p-1,p-1) & \text{if $p>2$ and $k_1-k_2\le 1$}\\
(2p-2,2p-2)  & \text{if $p>2$ and $k_1-k_2> 1$.}
\end{array}\right.
$$
We now study the filtration under $\Theta_{\uk}$. 
Henceforth we sometimes write $\Theta=\Theta_{\uk}$ for simplicity and accordingly, 
$\Theta^2$ means $\Theta_{\uk+(2,2)+m_{\uk}}\circ \Theta_{\uk}$. Similarly, for each integer 
$j\ge 1$, $\Theta^j$ is also inductively defined in that sense. 
For pairs $(a,b),(c,d)\in \Z^2$ with $a-b=c-d$. We write $(a,b)\le (c,d)$ if $b\le d$. 
The equality holds exactly when $(a,b)=(c,d)$. 
\begin{thm}\label{fil}
Suppose that $f\in M_{\uk}(N)$ is not weakly $p$-singular. Suppose further that $f$ is not 
identically zero on $S_{(0,0)}$ if $p>2$.  
Then 
\begin{enumerate}
\item $w(\Theta_{\uk}(f))\le w(f)+(2,2)+m_{\uk}$ for any $p$; 
\item suppose $p>2$ and then it holds   
\begin{itemize}
\item if $(k_1,k_2)=(k,k)$, the equality of {\rm (1)} and $\Theta_{\uk}(f)|_{S_{(0,0)}}\not\equiv 0$ 
hold unless $p|k(2k-1)$;
\item if $(k_1,k_2)=(k+1,k)$, the equality of {\rm (1)}  and $\Theta_{\uk}(f)|_{S_{(0,0)}}\not\equiv 0$ 
hold unless $p|(k+1)(2k-1)$;
\item if $k_1-k_2>1$, $\Theta_{\uk}(f)|_{S_{(0,0)}}\not\equiv 0$  and the equality of {\rm (1)}  always holds; 
\end{itemize}
\item $w(\Theta^{\frac{p+1}{2}}(f))=w(\Theta(f))$ if $p>2$ and 
$w(\Theta(f))=w(f)+(2,2)$ if $p=2$.  
\end{enumerate}
\end{thm}
\begin{proof}The inequality follows by definition. 
As explained in the proof of Theorem  \ref{non-vanishing}, it follows from 
the proof of Theorem 4.7-(2) in \cite{yam}  that 
$\Theta_{\uk}(f)$ is non-zero at some point in $S_{(0,0)}$ under the assumption on $\uk$. 
Further, under the assumption the equality holds in the case when $k_1-k_2\in\{0,1\}$ since the Hasse invariant is identically zero on $S_{(0,0)}$. 
When $k_1-k_2>1$, the local expansion (\ref{local-exp}) shows that 
$\Theta(f)$ is not a multiple of the Hasse invariant (otherwise, 
the right hand side of (\ref{local-exp}) has to be a multiple of $t_{11}t_{22}-t^2_{12}$ in 
the notation there). 

The second claim is a consequence of Theorem \ref{extension1}-(2) with the convention (\ref{conv1}) when $p>2$. 
When $p=2$, by Proposition \ref{local} and Lemma \ref{detA} it is easy to see 
$\Theta(H_{p-1}F)=H_{p-1}\Theta(F)$. The claim follows from this. 
\end{proof} 

\subsection{A new theta cycle}
Keep the notation in the previous subsection.
For each $f\in M_{\uk}(N)$ satisfying the assumption in Theorem \ref{fil}. 
The theta cycle of $f$ with respect to $\Theta_{\uk}$ is defined by 
\begin{equation}\label{cycle}
{\rm Cyc}(f):=
\left\{\begin{array}{ll}
(w(\Theta(f)),w(\Theta^2(f)),\ldots,w(\Theta^{\frac{p-1}{2}}(f))) & \text{if $p>2$}\\
(w(\Theta(f))) & \text{if $p=2$}
\end{array}\right.
\end{equation}
For the second projection $p_2:\Z^2\lra \Z,\ (x,y)\mapsto y$, we also define 
\begin{equation}\label{cyclep2}
p_2({\rm Cyc}(f)):=
\left\{\begin{array}{ll}
(p_2(w(\Theta(f))),p_2(w(\Theta^2(f))),\ldots,p_2(w(\Theta^{\frac{p-1}{2}}(f)))) & \text{if $p>2$}\\
(p_2(w(\Theta(f)))) & \text{if $p=2$}
\end{array}\right.
\end{equation}
Notice that any $\Theta^j(f)$ in the above cycle is not identically zero by 
Theorem \ref{non-vanishing}.
Since $w(\Theta(f))=w(\Theta^{\frac{p+1}{2}}(f))$ if $f$ is a Hecke eigenform and $p>2$ under our convention (\ref{conv1}) 
and $w(\Theta^2(f))=w(\Theta(f))$ if $p=2$, actually it makes up a ``cycle" in 
some sense.  

When $p=2$ or $k_1-k_2>1$, the theta cycle is easily computed by 
Theorem \ref{fil}. 
Therefore, we focus on the case when $k_1-k_2\in\{0,1\}$ and $p>2$.
\begin{dfn}\label{low-pt}Assume $p>2$. 
Let $f$ be an element in $M_{\uk}(N)$ satisfying the assumption in Theorem \ref{fil}.  
 Suppose $(k_1,k_2)=(k,k)$ or $(k+1,k)$. Put $r=k_1-k_2$ so that $r\in \{0,1\}$.  
\begin{enumerate}
\item  
We say $\Theta^i(f)$ is a low point of the first type $($resp. the second type$)$ if 
 $p_2(w(\Theta^{i-1}(f)))+r\equiv 0\ {\rm mod}\ p$ $($resp. $2p_2(w(\Theta^{i-1}(f)))-1\equiv 0\ {\rm mod}\ p)$. 
If $f_i:=\Theta^{c_i}(f)$ is a low point  for some integer $c_i>0$, then the number $c_i-1$ means one of times we add $(p+1,p+1)$ to 
$w(f)$. We say $c_i$ the low number of the low point $\Theta^{c_i}(f)$. We say $c_i$ the low number for $f_i$. 
We write $c_i=c^{(1)}_i$ $($resp. $c_i=c^{(2)}_i)$ if 
the low point is of the first type $($resp. the second type$)$. 

\item We define the number $b_i$ such that 
\begin{equation}\label{bi}
b_i(p-1)=p_2(w(\Theta^{c_i-1}f))+(p+1)-
p_2(w(\Theta^{c_i}f))
\end{equation}
which means the amount falling the filtration at the low point $f_i$ with the next application of 
$\Theta$. 
We say $b_i$ the jumping number of the low point $\Theta^{c_i}(f)$.  
As in $c_i$, we also write $b_i=b^{(1)}_i$ or $b_i=b^{(2)}_i$ according to the first type or the second type respectively. 
\end{enumerate}
\end{dfn}
We illustrate the notion of low points as below. 
$$
\xymatrix{
f  \ar[r] & \Theta(f) \ar[r]&  \cdots \ar[r] & \ar@//[dlll]  \Theta^{c_1^{(j_1)}-1}(f)   \\
f^{(j_1)}_1:=\Theta^{c_1^{(j_1)}}(f)  \ar[r] & \Theta(f^{(j_1)}_1) \ar[r]&  \cdots \ar[r] & \ar@//[dlll]  \Theta^{c_2^{(j_2)}-1}(f^{(j_1)}_1)   \\
\cdots  \ar[r] &\cdots \ar[r]&  \cdots \ar[r] & \cdots  \\   
} 
$$
where $j_1,j_2,\ldots\in\{1,2\}$. 
The variant of the filtration goes as follows:
$$\small{
\xymatrix{
k:=p_2(w(f))  \ar[r] & p_2(w(\Theta(f)))=k+(p+1) \ar[r]&  \cdots \ar[r] & 
\ar@//[dlll]_{\text{The weight falls $b^{(j_1)}_1(p-1)$}}   p_2(w(\Theta^{c_1^{(j_1)}-1}(f)))=k+(c_1^{(j_1)}-1)(p+1)   \\
p_2(w(f^{(j_1)}_1))  \ar[r] & p_2(w(\Theta(f^{(j_1)}_1))) \ar[r]&  \cdots \ar[r] & 
\ar@//[dlll]_{\text{The weight falls $b^{(j_2)}_2(p-1)$}}   
p_2(w(\Theta^{c_2^{(j_2)}-1}(f^{(j_1)}_1)))   \\
\cdots  \ar[r] &\cdots \ar[r]&  \cdots \ar[r] & \cdots.  \\   
}}
$$    
Then the number $c_1^{(j_1)}$ means 
that the 
first number such that 
\begin{equation}\label{p2c}
p_2(w(\Theta^{j}(f)))=p_2(w(f))+j(p+1)\  {\rm for}\ 0\le j\le c_1^{(j_1)}-1
\end{equation}
 and 
\begin{equation}\label{lowpoints}
\left\{\begin{array}{cc}
p_2(w(\Theta^{c_1^{(j_1)}-1}(f)))+r\equiv 0\ {\rm mod}\ p & \text{if $j_1=1$ 
(a low point of the first type)}\\
2p_2(w(\Theta^{c_1^{(j_1)}-1}(f)))-1\equiv 0\ {\rm mod}\ p & \text{if $j_1=2$  
(a low point of the second type)}.
\end{array}
\right.
\end{equation} 
Let $\{c^{(j_i)}_i\}_{i=1}^s$ (resp. $\{b^{(j_i)}_i\}_{i=1}^s$) be the collection of all low numbers (jumping numbers) for $f$. 
We define $f^{(j_i)}_i,1\le i\le s$ inductively such that  
$f^{(j_{i+1})}_{i+1}=\Theta^{c^{(j_{i+1})}_{i+1}}(f^{(j_i)}_i)$ and $f^{(j_1)}_1:=\Theta^{c_1^{(j_1)}}(f)$. 
Since the length of the theta cycle of $f$ is $\ds\frac{p-1}{2}$, one has 
\begin{equation}\label{sumc}
\ds\sum_{i=1}^s c^{(j_i)}_i=\ds\frac{p-1}{2}.
\end{equation}
The total amount of the varying weights in the theta cycle is $(p+1)\ds\frac{(p-1)}{2}$. 
It follows from this that $\ds\sum_{i=1}^s b^{(j_i)}_i(p-1)=(p+1)\frac{(p-1)}{2}.$ Hence we have 
\begin{equation}\label{sumb}
\ds\sum_{i=1}^s b^{(j_i)}_i=\frac{p+1}{2}. 
\end{equation}
Further, by definition we have
\begin{equation}\label{c-formula2}
p_2(w(\Theta^{c^{(j_{i+1})}_{i+1}-1}f^{(j_i)}_{i}))=p_2(w(f^{(j_i)}_{i}))+
(c^{(j_{i+1})}_{i+1}-1)(p+1),\ 
j_i,j_{i+1}\in\{1,2\}.
\end{equation} 
Further, the equations (\ref{bi}) and (\ref{lowpoints}) turn out to be 
\begin{equation}\label{bi2}
b^{(j_i)}_i(p-1)=p_2(w(\Theta^{c^{(j_i)}_i-1}f^{(j_{i-1})}_{i-1}))+(p+1)-
p_2(w(f^{(j_i)}_i))
\end{equation}
and for each $1\le i\le s$, 
\begin{equation}\label{lowpoints2}
\left\{\begin{array}{cc}
p_2(w(\Theta^{c_i^{(j_i)}-1}(f^{(j_{i-1})}_{i-1})))+r\equiv 0\ {\rm mod}\ p & \text{if $j_i=1$ (a low point of 
the first type)}\\
2p_2(w(\Theta^{c_i^{(j_i)}-1}(f^{(j_{i-1})}_{i-1})))-1\equiv 0\ {\rm mod}\ p & \text{if $j_i=2$  (a low point of 
the first type)}
\end{array}
\right.
\end{equation}  
respectively.

\subsubsection{The case when $r=0$}
In this case, the theta cycle ${\rm Cyc}(f)$ for $f\in M_{\uk}(N,\bF_p)$ with 
$k_2\ge 2$ was computed in Section 6.0.1-6.0.2 in \cite{yam}. 

\subsubsection{The case when $r=1$}
The computation is quite similar  to the case when $r=0$ but we give the details 
for reader's convenience. 

Let $f$ be a non-zero element in $M_{\uk}(N,\bF_p)$ with 
$k_2\ge 2$. 
It follows from (\ref{lowpoints2}) and (\ref{bi2}) that 
\begin{equation}\label{bi-cong}
\left\{\begin{array}{ll}
p_2(w(f^{(j_i)}_i))\equiv  b^{(j_i)}_i  \ {\rm mod}\ p & \text{if $j_i=1$}\\
p_2(w(f^{(j_i)}_i))\equiv  b^{(j_i)}_i+\frac{p+3}{2}\  {\rm mod}\ p & \text{if $j_i=2$}.
\end{array}
\right.
\end{equation} 
Further, the condition (\ref{lowpoints2}) and (\ref{c-formula2}) yield 
\begin{equation}\label{ci-cong}
p_2(w(f^{(j_i)}_i))+c^{(j_{i+1})}_{i+1}\equiv 
\left\{\begin{array}{cc}
 -1\ {\rm mod}\ p & \text{if $j_{i+1}=1$}\\
 \frac{p+1}{2}\  {\rm mod}\ p & \text{if $j_{i+1}=2$}.
\end{array}
\right.
\end{equation} 
It follows from (\ref{bi-cong}) and (\ref{ci-cong}) that 
\begin{equation}\label{cong-rel}
\left\{\begin{array}{ll}
\text{Case 1} & c^{(1)}_{i+1}+b^{(1)}_i \equiv 0\ {\rm mod}\ p \\
\text{Case 2} & c^{(2)}_{i+1}+b^{(1)}_i \equiv \frac{p+3}{2}\ {\rm mod}\ p \\
\text{Case 3} & c^{(1)}_{i+1}+b^{(2)}_i \equiv \frac{p-3}{2}\ {\rm mod}\ p \\
\text{Case 4} & c^{(2)}_{i+1}+b^{(2)}_i \equiv 0\ {\rm mod}\ p 
\end{array}
\right.
\end{equation} 
This also shows $c^{(j_{i+1})}_{i+1}+b^{(j_i)}_i\ge \min\{\frac{p-3}{2},3\}$ for 
any $1\le i\le s-1$ and $j_i,j_{i+1}\in \{1,2\}$. 
If $s\ge 2$, it has to be $s=2$ since 
$$p=\ds\sum_{i=1}^s (b^{(j_i)}_i+c^{(j_i)}_i)=
b^{(j_s)}_s+c^{(j_1)}_1+\sum_{i=1}^{s-1} (b^{(j_i)}_i+c^{(j_{i+1})}_{i+1})$$ 
and $p\ge 3$.  
Therefore, the number of low points $s$ is less than or equal to $2$. 

Now assume $f$ is a Hecke eigen form for $\mathbb{T}$ and it is non semi-ordinary in the sense of \cite[Section 6,1]{yam} in terms of 
Hecke eigenvalues at $p$ which is easily extended to the vector valued case by using the formula of \cite[p.173]{Arakawa} 
(notice that we have assumed $k_2\ge 2$ in this section so that the Hecke operators at $p$ are well defined 
(see \cite[p.10, Remark 2.1]{yam}) 

Then $p_2(w(\Theta^{\frac{p-1}{2}}f))=p_2(w(f))$. 
Hence, a jump happens at least one time and thus $s\ge 1$. Further, 
the last jump necessarily happens at $w(\Theta^{\frac{p-3}{2}}f)$ to conclude 
$w(\Theta^{\frac{p-1}{2}}f)=w(f)$. 
Further, since $w(f)$ appears in ${\rm Cyc}(f)$, the next step $w(\Theta(f))$ is 
automatically outputted in the computation below. 
This is a phenomena in the case of non semi-ordinary. 

When $s=1$, by (\ref{sumc}) and (\ref{sumb}) we have 
$$c^{(j_1)}_1=\frac{p-1}{2},\ b^{(j_1)}_1=\frac{p+1}{2},\ j_1\in\{1,2\}.$$
Put $k:=p_2(w(f))=ap+k_0$ with $a\in \Z_{\ge 0}$ and $1\le k_0\le p$. 
When $j_1=1$, 
since $p_2(w(\Theta^{c_1^{(1)}-1}f))=k+(c^{(1)}_1-1)(p+1)\equiv -1$ mod $p$, 
$k_0\equiv \frac{p+1}{2}$ mod $p$. 
Hence $k_0=\frac{p+1}{2}$. 
As a check, by formula (\ref{bi2}) 
we have $p_2(w(\Theta^{c_1^{(1)}}f))=k+(c^{(1)}_1-1)(p+1)+(p+1)-b^{(1)}_1(p-1)=k$.  
Similarly, when $j_1=2$, we have $k_0=2$. 
In either of cases, we have  
$$p_2({\rm Cyc}(f))=(k+(p+1),\ldots,\overbrace{k+\frac{(p-3)}{2}(p+1)}^{c^{(j_1)}_1-1=\frac{p-3}{2}},k),\ 
k=ap+k_0$$ 
with $k_0=\frac{p+1}{2}$ if $j_1=1$ and $k_0=2$ if $j_1=2$.

Henceforth we assume $s=2$ (this implies $p\ge 5$ since the length of the theta cycle is 
$\frac{p-1}{2}$). Since 
$$p=\ds\sum_{i=1}^2 (b^{(j_i)}_i+c^{(j_i)}_i)=(b^{(j_2)}_2+c^{(j_1)}_1)+(b^{(j_1)}_1+c^{(j_{2})}_{2})
\ge 2+(b^{(j_1)}_1+c^{(j_{2})}_{2}).$$
Combining it with the congruence relation (\ref{cong-rel}) we have 
$$(j_1,j_2)=(1,2)\ {\rm or}\ (2,1).$$
Put $k:=p_2(w(f))=ap+k_0$ with $a\in \Z_{\ge 0}$ and $1\le k_0\le p$. 
When $(j_1,j_2)=(1,2)$, we have $b^{(1)}_1+c^{(2)}_{2}=\frac{p+3}{2}$. 
Since $p_2(w(\Theta^{c_1^{(1)}-1}f))=k+(c^{(1)}_1-1)(p+1)\equiv -1$ mod $p$, 
$k_0+c^{(1)}_1\equiv 0$ mod $p$. 
It follows from $2\le k_0+c^{(1)}_1\le p+\frac{p-1}{2}<2p$ that $k_0+c^{(1)}_1=p$. 
Then we have 
$$c^{(1)}_1=p-k_0,\ c^{(2)}_2=k_0-\frac{p+1}{2},\ b^{(1)}_1=p+2-k_0,\ 
b^{(2)}_2=k_0-\frac{p+3}{2}.$$
Since these integers are positive integers, we should have 
$\frac{p+5}{2}\le k_0\le p-1$ and also $p\ge 7$. Then the theta cycle is computed 
as 
$$p_2({\rm Cyc}(f))=(k+(p+1),\ldots,\overbrace{k+(p-1-k_0)(p+1)}^{c^{(1)}_1-1=p-1-k_0},
k^{(1)}_1,k^{(1)}_1+(p+1),\ldots,\overbrace{k^{(1)}_1+(k_0-\frac{p+3}{2})(p+1)}^{c^{(2)}_2-1=k_0-\frac{p+3}{2}},k)$$
where 
$$k^{(1)}_1:=k+(c^{(1)}_1-1)(p+1)-b^{(1)}_1(p-1)=k-p+1-2k_0.$$ 

When $(j_1,j_2)=(2,1)$, we have $b^{(2)}_1+c^{(1)}_{2}=\frac{p-3}{2}$. 
A similar argument shows $k_0+c^{(2)}_1=\frac{p+3}{2}$. Hence we have 
$$c^{(2)}_1=\frac{p+3}{2}-k_0,\ c^{(1)}_2=k_0-2,\ b^{(2)}_1=\frac{p+1}{2}-k_0,\ 
 b^{(1)}_2=k_0$$
 with $3\le k_0\le \frac{p-1}{2}$ and also $p\ge 7$. 
Then the theta cycle is computed 
as 
$$p_2({\rm Cyc}(f))=(k+(p+1),\ldots,\overbrace{k+(\frac{p+1}{2}-k_0)(p+1)}^{c^{(2)}_1-1=
\frac{p+1}{2}-k_0},k^{(2)}_1,k^{(2)}_1+(p+1),\ldots,\overbrace{k^{(2)}_1+(k_0-3)(p+1)}^{c^{(1)}_2-1=
k_0-3},k)$$ 
where 
$$k^{(2)}_1:=k+(c^{(2)}_1-1)(p+1)-b^{(2)}_1(p-1)=k+p+1-2k_0.$$

Summing up, we have obtained the following result for the theta cycle:
\begin{thm}\label{r1-non-semi}Assume $(k_1,k_2)=(k+1,k)$ with $k\ge 2$.  Let $f\in M_{\uk}(N)$ satisfying the assumption in Theorem \ref{fil}. Assume $f$ is non semi-ordinary. 
Put $k:=p_2(w(f))=ap+k_0$ with $a\in \Z_{\ge 0}$ and $1\le k_0\le p$. 
\begin{enumerate}
\item 
Let $k^{(1)}_1:=k-p+1-2k_0$ and $k^{(2)}_1:=k+p+1-2k_0$.
If $p\ge 5$ and the number of low points is equal to $2$, then $p_2({\rm Cyc}(f))$ is given by    
$$
\begin{array}{ll}
(k+(p+1),\ldots,k+\frac{(p-3)}{2}(p+1),k) & 
\text{if $k_0=\frac{p+1}{2}$ with $j_1=1$ or $k_0=2$ with $j_1=2$} \\
(k+(p+1),\ldots,k+(p-1-k_0)(p+1),\\
  k^{(1)}_1,k^{(1)}_1+(p+1),\ldots,
k^{(1)}_1+(k_0-\frac{p+3}{2})(p+1),k) & \text{if $p\ge 7$ and $k_0\in [\frac{p+5}{2},p-1]$} \\
(k+(p+1),\ldots,k+(\frac{p+1}{2}-k_0)(p+1),\\
 k^{(2)}_1,k^{(2)}_1+(p+1),\ldots,k^{(2)}_1+(k_0-3)(p+1),k) & \text{if $p\ge 7$ and $k_0\in [3,\frac{p-1}{2}]$}. 
\end{array}
$$
Note that the pair $(k_0,p)$ which does not satisfy any of the conditions can not occur when 
the number of low points is equal to $2$. It is the same when $k^{(1)}_1\le 0$. 
\item If $p\ge 3$ and the number of low points is equal to $1$, 
then $$p_2({\rm Cyc}(f))=(k+(p+1),\ldots,\overbrace{k+\frac{(p-3)}{2}(p+1)}^{c^{(j_1)}_1-1=\frac{p-3}{2}},k)$$
where  $k_0=\frac{p+1}{2}$ if $j_1=1$ and $k_0=2$ if $j_1=2$. 
\item 
If $p=2$, 
$$p_2({\rm Cyc}(f))=(p_2(w(f))+2).$$
\end{enumerate}
\end{thm}
When $s=2$, the lows points happens so that 
\begin{enumerate}
\item the low point of the second kind comes after   
 the low point of the first kind in which case $(j_1,j_2)$ is $(1,2)$, and 
\item the low point of the first kind comes after   
 the low point of the second kind in which case $(j_1,j_2)$ is $(2,1)$.
\end{enumerate}
Therefore, once $w(f)$ is given, one can check which case of two happens.  
\begin{cor}\label{wt-cong-imply}
Keep the notation being in Theorem \ref{r1-non-semi}. 
If $f$ is non semi-ordinary and $p$ is odd, then $p_2(w(f))\not\equiv 1,\frac{p+3}{2},p$ if 
the number of low points is equal to $2$.
\end{cor}

Finally we discuss when $f$ is semi-ordinary. In this case 
we do not know if $w(f)$ appears in ${\rm Cyc}(f)$. 
Therefore, we need to observe the first step $w(\Theta(f))$ to start the theta cycle. 
Notice that $\Theta(f)$ is non semi-ordinary in the sense of \cite[Section 6,1]{yam}.  
This can be checked by using the formula of \cite[p.173]{Arakawa}. Thus, we plugin 
the non semi-ordinary case with 
$\Theta(f)$.  
This is a bit cumbersome part in the case of semi-ordinary. 

Let $k:=p_2(w(f))=ap+k_0$.  
Suppose $p\nmid (k_0+1)(2k_0-1)$. Then $p_2(w(\Theta(f)))=k+p+1=(a+1)p+(k_0+1)$. 
Let $g:=\Theta(f)$. Then $w(\Theta^{\frac{p-1}{2}}g)=w(g)$.
Applying  Theorem \ref{r1-non-semi}  to  $g:=\Theta(f)$ and $p_2(w(\Theta(f)))$, 
we have the following:
 \begin{thm}\label{r1-semi}Assume $(k_1,k_2)=(k+1,k)$ with $k\ge 2$.  Let $f\in M_{\uk}(N)$ satisfying the assumption in Theorem \ref{fil}. Assume $f$ is semi-ordinary. 
Put $k:=p_2(w(f))=ap+k_0$ with $a\in \Z_{\ge 0}$ and $1\le k_0\le p$. 
Let $k^{(1)}_1:=k+(p+1)-2k_0$ and $k^{(2)}_1:=k+(p+1)+2p-2k_0$. 
Suppose $p\nmid (k_0+1)(2k_0-1)$. 
If $p>2$, then $p_2({\rm Cyc}(f))$ satisfies either of the followings  
$$
\begin{array}{ll}
(k+(p+1),\ldots,k+\frac{(p-1)}{2}(p+1)) & 
\text{if $k_0=1$ or $\frac{p-1}{2}$} \\
(k+(p+1),\ldots,k+(p-2-k_0)(p+1),\\
  k^{(1)}_1,k^{(1)}_1+(p+1),\ldots,
k^{(1)}_1+(k_0-\frac{p+1}{2})(p+1)) & \text{if $p\ge 7$ and $k_0\in [\frac{p+3}{2},p-2]$} \\
(k+(p+1),\ldots,k+(\frac{p-1}{2}-k_0)(p+1),\\
 k^{(2)}_1,k^{(2)}_1+(p+1),\ldots,k^{(2)}_1+(k_0-2)(p+1)) & \text{if $p\ge 7$ and $k_0\in [2,\frac{p-3}{2}]$} \\
 (k+(p+1),k+2(p+1),\ldots,k+\frac{(p-1)}{2}(p+1))  & \text{otherwise}.
\end{array}
$$
If $p=2$, $p_2({\rm Cyc}(f))=(p_2(w(f))+2)$. 
\end{thm}
The remaining cases are $p|(k_0+1)$ or $p|(2k_0-1)$ when $p>2$. 
In this case, once we obtain $b\in \Z_{\ge 0}$ such that 
$p_2(w(\Theta(f)))=p_2(w(f))+(p+1)-b(p-1)$, then we may apply Theorem \ref{r1-semi} to 
$p_2(w(\Theta(f)))$. Notice $p_2(w(\Theta(f)))$ is congruent to $b$ or $\frac{p+3}{2}+b$. 
We would try to determine $b$ somewhere else. 

\section{Classical Serre weights for $p>2$}\label{CSW1} 
In this section we follow the terminology for $p$-adic Hodge theory in 
\cite[Section 1]{BGGT}.  
Assume $p>2$. For each positive integer $n$, we denote by 
$\omega_n:G_{\Q_{p^n}}\lra \bF^\times_p$ the fundamental character of level $n$ where 
$\Q_{p^n}$ is a unique unramified extension of $\Q_p$ of degree $n$. 
Note that $\omega_1=\ve$ where, by abusing notation, $\ve$ stands for 
the mod $p$ cyclotomic character of $G_{\Q_p}$.

\begin{dfn}\label{SSW} 
For a given mod $p$ Galois representation $\br:G_\Q\lra \GSp_4(\bF_p)$, 
Let ${\rm SW}(\br)$ be the subset of $\Z^2_{\ge 0}\times \Z$ consisting of 
all triples $(k_1,k_2,w)$ with $k_1\ge k_2\ge 3$ satisfying
\begin{itemize}
\item there exists a potentially diagonalizable, crystalline lift $\rho$ of $\br|_{G_{\Q_p}}$ 
which takes the values in $GSp_4$ such that $\rho$ has regular Hodge-Tate weights:
$${\rm HT}(\rho)=\{k_1+k_2-3+w,k_1-1+w,k_2-2+w,w\}.$$
\end{itemize}
\end{dfn}
 
According to Section 10 of \cite{yam2} we give an explicit description of 
the classical Serre weight $$(k_1(\br),k_2(\br),w(\br))\in {\rm SW}(\br)$$ for 
a given mod $p$ Galois representation $\br:G_\Q\lra \GSp_4(\bF_p)$. 
Put $\br_p:=\br|_{G_{\Q_p}}$. 
For a character $\ochi:G_{\Q_p}\lra \bF^\times_p$, a class of the 
Galois cohomology $H^1(G_{\Q_p},\ochi)$ is said to be 
tr\`es ramifi\'ee (ramified) if $\ochi=\ve$ and it is a tr\`es ramifi\'ee (ramified) class in the sense of Definition 9.10 of \cite{yam2}. If  a class is neither tr\`es ramifi\'ee nor ramified, 
then we say it peu ramifi\'ee.  

Henceforth the characters $\op_i:G_{\Q_p}\lra \bF^\times_p,\ 0\le i\le 2$ always mean 
finite unramified characters. 
By Proposition 7.2 of \cite{yam} we have five types of $\br_p$ whose 
images take the values in $\GSp_4(\bF_p)$:
\begin{enumerate}
\item (Borel ordinary case) 
$$\br_p\simeq \op_0\ve^c \otimes
\begin{pmatrix}
\br_1 & B \\
0_2 & \op_1\ve^{a+b}\br^\ast_1 
\end{pmatrix},\ \br_1=\begin{pmatrix}
 \op_1\ve^{a+b} & \ot_0 \\
0 & \op_2\ve^{a}  
\end{pmatrix},\ B\in H^1(\Q_p,(\op_1\ve^{a+b})^{-1}{\rm Sym}^2(\br_1));$$ 
\item (Siegel ordinary case)  
$$\br_p\simeq \op_0\ve^c\otimes
\begin{pmatrix}
 \op_2\ve^{a+b} & \ot_1 & \ot_3 \\ 
0 & \br_1& \ot_2 \\
0 & 0& 1 
\end{pmatrix}
$$
where $\br_1\simeq \op_1\otimes{\rm Ind}^{G_{\Q_p}}_{G_{\Q_{p^2}}}\omega^{b+ap}_2$ with $0\le b<a\le p$,  
$0\le c\le p-2$, and  $\op_1,\op_2$ satisfy $\op^2_1=\det(\br_1)\ve^{-(a+b)}=\op_2$; 
\item (Klingen ordinary case) 
$$\br_p\simeq 
\begin{pmatrix}
\br_1 & \ast \\ 
0_2 & \op_0\ochi_p^c \br_2
\end{pmatrix}$$ 
where $\br_1\simeq \op_1\otimes{\rm Ind}^{G_{\Q_p}}_{G_{\Q_{p^2}}}\omega^{b+ap}_2$ with $0\le b<a\le p$, $c\in \Z$ and 
$\br_2\simeq \br^\vee_1$;
\item (Endoscopic case) 
$$\br_p\simeq \br_1\oplus \br_2$$ where $\br_i\simeq \op_i\otimes  {\rm Ind}^{G_{\Q_p}}_{G_{\Q_{p^2}}}\omega^{b_i+pa_i}_2$ with 
$0\le b_i<a_i<p$ for $i=1,2$ and $\det(\br_1)=\det(\br_2)$. Further, we may assume 
$\br_1|_{I_{\Q_p}}\not\simeq \br_2|_{I_{\Q_p}}$ since otherwise it is essentially subsumed into the Klingen ordinary case. Here $I_{\Q_p}$ is the inertia subgroup of $G_{\Q_p}$.  
\item (Irreducible case) 
$$\br_p\simeq \op_0\ve^c\otimes {\rm Ind}^{G_{\Q_p}}_{G_{\Q_{p^4}}}\omega^{a}_4,\ a\in \Z$$
where $a\not\equiv 0\ {\rm mod}\ (p^4-1)/(p^i-1)$ for $i=1,2$ but 
$a\equiv 0\ {\rm mod}\ p+1$. 
\end{enumerate}  
In what follows we will define the classical Serre weights and it is defined according to 
how $\br_p$ is lifted to a crystalline (potentially diagonalizable) representation of $G_{\Q_p}$. 
The classical weights consist of a triple $(k_1(\br),k_2(\br),w(\br))$ of integers with 
$k_1(\br)\ge k_2(\br)\ge 3,\ w(\br)\in \Z$. 
\subsection{Local Galois cohomologies}
This is a preliminary to define the classical Serre weights. 
For each character $\ochi:G_{\Q_p}\lra \bF^\times_p$ and each class $\alpha\in 
H^1(\Gp,\ochi)$,   
put 
\begin{equation}\label{rchi}
r_{\ochi,\alpha}:=
\left\{\begin{array}{cl}
p-1 & \mbox{$\alpha$ is tr\`es ramifi\'ee or ramified} \\
0   & \mbox{otherwise}
\end{array}
\right.
\end{equation}
(see Definition 9.14 of \cite{yam2}). 
For integers $i,j$ and a prime $p$, we define an element of $\Z^2$ by 
$$\delta^p_{ij}:=
\left\{\begin{array}{cl}
(p-1,p-1) & \mbox{if $i=j$} \\
(0,0)   & \mbox{otherwise.}
\end{array}
\right.
$$
Let $\F$ be a finite extension of $\F_p$. 
Let $\overline{M}$ be a finite dimensional representation over $\bF$ 
of $G_{\Q_p}$ and $M$ be a lift of $\overline{M}$ to some $p$-adically integral 
ring $\O$ whose residue field is $\F$. Here $M$ is considered as an $\O[G_{\Q_p}]$-module. 
The local Galois cohomology $H^i(\Gp,M)$ is a finite $\O$-module and may have torsion 
elements. We denote by $H^i(\Gp,M)_{{\rm tf}}$ the quotient of $H^i(\Gp,M)$ by 
all torsion elements. 
Suppose $\F(\ve)\subset \bM$ as an $\F[G_{\Q_p}]$-module. Then it yields  
$H^1(G_{\Q_p},\F(\ve))\lra H^1(G_{\Q_p},\bM)$. 
Let $M^\ast:=M^\vee(1)=M^\vee(\ve)$ and $\langle \ast,\ast\rangle$ be the perfect pairing 
on $H^1(G_{\Q_p},\bM)\times H^1(G_{\Q_p},\bM^\ast)$ defined by the local Tate duality.  
The natural surjection $\bM^\ast \lra (\F(\ve))^\ast=\F$ induces 
$H^1(G_{\Q_p},\bM^\ast)\stackrel{\iota_{\bM^\ast}}{\lra} H^1(G_{\Q_p},\F)={\rm Hom}(\Q^\times_p/(\Q^\times_p)^p,\F)$. 
Let $e_{{\rm ur}}$ (resp. $e_{{\rm ram}}$) be a non-zero element in the RHS such that $e_{{\rm ur}}(p)=1$ 
and zero on the units (resp. $e_{{\rm ram}}(1+p)=1$). 
Since $p>2$, they make up a basis of $H^1(G_{\Q_p},\F)$. 
For $\spadesuit\in \{{\rm ur},\ {\rm ram}\}$, pick a lift $\widetilde{e}_\spadesuit$ of 
$e_\spadesuit$ to 
$H^1(G_{\Q_p},\bM^\ast)$ under $\iota_{\bM^\ast}$. 
If there does not exist any lift, we put $\widetilde{e}_\ast:=0$. 
For each class $\alpha$ of $H^1(G_{\Q_p},\bM)$ we define 
\begin{equation}\label{marking1}
a_{{\rm ur}}(\alpha):=\langle \alpha, \widetilde{e}_{{\rm ur}} \rangle,\ 
a_{{\rm ram}}(\alpha):=\langle \alpha, \widetilde{e}_{{\rm ram}} \rangle
\end{equation}
Similarly, if $\F\subset \bM$ as an $\F[G_{\Q_p}]$-module, then we have 
$$H^1(G_{\Q_p},\bM^\ast)\lra H^1(G_{\Q_p},\F(\ve))\simeq \Q^\times_p/(\Q^\times_p)^p$$
where the above last isomorphism is the Kummer map. 
Then we have the dual basis  $e^\ast_{{\rm ur}},\ e^\ast_{{\rm ram}}$ of 
$H^1(G_{\Q_p},\F(\ve))$. 
For each class $\alpha$ of $H^1(G_{\Q_p},\bM)$ we also define 
\begin{equation}\label{marking2}
a_{{\rm ur},\ast}(\alpha):=\langle \alpha, \widetilde{e}^\ast_{{\rm ur}} \rangle,\ 
a_{{\rm ram},\ast}(\alpha):=\langle \alpha, \widetilde{e}^\ast_{{\rm ram}} \rangle
\end{equation}
where lifts $\widetilde{e}^\ast_{{\rm ur}},\ \widetilde{e}^\ast_{{\rm ram}}$ of 
$e^\ast_{{\rm ur}},\ e^\ast_{{\rm ram}}$ to 
$H^1(G_{\Q_p},\bM^\ast)$ 
are similarly defined as above. 
\begin{dfn}\label{marking}Keep the notation being as above. Assume $p$ is odd. 
For each class $\alpha\in H^1(G_{\Q_p},\bM)$, the pair 
$(a_{{\rm ur}}(\alpha),a_{{\rm ram}}(\alpha))$ defined in $($\ref{marking1}$)$ 
is said to be the marking of $\alpha$ with respect to 
$\{\widetilde{e}_{{\rm ur}},\widetilde{e}_{{\rm ram}}\}$. 

Similarly, the pair 
$(a_{{\rm ur},\ast}(\alpha),a_{{\rm ram},\ast}(\alpha))$ defined in $($\ref{marking2}$)$ 
is said to be the marking of $\alpha$ with respect to $\{\widetilde{e}^\ast_{{\rm ur}},\widetilde{e}^\ast_{{\rm ram}}\}$. 
\end{dfn}

\subsection{Borel ordinary case}\label{bo} 
Recall $B=\begin{pmatrix}
b_2 & b_1 \\
b_3 & b_2
\end{pmatrix}$
 belongs to 
$H^1(\Q_p,(\op_1\ve^{a+b})^{-1}{\rm Sym}^2(\br_1))$ where 
\begin{equation}\label{sym2}
(\op_1\ve^{a+b})^{-1}{\rm Sym}^2(\br_1)=
\begin{pmatrix}
\op_1\ve^{a+b} & 2\ot_0 & \ot^2_0 \\
0 &   \op_2\ve^a & \op^{-1}_1\op_2 \ve^{-b}\ot_0 \\
0 & 0 & \op^{-1}_1\op^2_2 \ve^{a-b}
\end{pmatrix}.
\end{equation}
Put 
\begin{equation}\label{chi123}
\ochi_1:=\op_1\ve^{a+b} ,\ \chi_2:=\op_2\ve^a ,\ \ochi_3:=\op^{-1}_1\op^2_2 \ve^{a-b}, 
\end{equation}
$\br_3:=(\op_1\ve^{a+b})^{-1}{\rm Sym}^2(\br_1))$, and 
$\br_2:=\begin{pmatrix}
  \ochi_2 & \op^{-1}_1\op_2 \ve^{-b}\ot_0 \\
               0 & \ochi_3
\end{pmatrix}$ for simplicity. 
The exact sequence $0\lra \ochi_1\lra \br_3\lra \br_2\lra 0$
yields an exact sequence 
\begin{equation}\label{exact-galois}
H^1(G_{\Q_p},\ochi_1)\lra 
H^1(G_{\Q_p},\br_3)\lra H^1(G_{\Q_p},\br_2). 
\end{equation}
By choosing a conjugation of $\br_3$ if necessary,  we have 
a homomorphism $H^1(\Gp,\br_3)\lra H^1(\Gp,\ochi_i)$ which sends 
$B$ to $b_i$ for each $1\le i\le 3$. Regarding the usage of (\ref{rchi}), 
$b_i$ is considered in this manner whenever we involve $r_{\ochi_i,b_i}$ in the discussion 
below. 
\subsubsection{The case of the trivial extension}\label{trivial-ext} 
First we consider the case when $\ot_0=0$.  
We have $H^1(\Gp,\br_3)=\ds\bigoplus_{1\le i\le 3}H^1(\Gp,\ochi_i)$. 
By Proposition 9.12 of \cite{yam2} one can lift $b_i$ to a crystalline extension in 
$H^1(\Gp,\chi_i)$ for some crystalline lift $\chi_i$ of $\ochi_i$. 
Then we define $w(\br)=c$ and 
$$(k_1(\br),k_2(\br))=
\left\{
\begin{array}{ll}
(1,2)+(a+\ds\max_{1\le i\le 3}\{r_{\ochi_i,b_{i}}\},b)+\delta^p_{b0} & 
\text{if $a>b$} \\
(1,2)+(a+p-1,a)+\delta^p_{a0} & 
\text{if $a=b$} \\
(1,2)+(a+p-1+r_{\ochi_3,b_3},b)& 
\text{if $a<b$}.
\end{array}\right.
$$
We remark that when $a<b$, the ramified case does not happen by definition. 
\subsubsection{The case of the non-trivial peu ramifi\'ee extensions}
Let us keep the notation in the previous subsection. To define the classical Serre weights, 
we apply the argument in Section 9.3.3 in \cite{yam2}.  
In this case, since $\ot_0$ is peu ramifi\'ee  one can lift $\br_2$ to a (potentially diagonalizable) crystalline lift $\rho_{2,m_2,m_3}$
of Hodge-Tate weights $\{m_1,m_2\}$ with $m_2>m_3>0$ satisfying $(m_2,m_3)\equiv 
(a,a-b)$ modulo $p-1$. 
Let $B_2$ be the image of $B$ under (\ref{exact-galois}). 
Applying Proposition 9.16 of \cite{yam}, we first lift $B_2$ to an element $\widetilde{B}_2$ 
in $H^1(\Q_p,\rho_{2,s_1,s_2})$ by suitably choosing $\rho_{2,s_1,s_2}$ and 
also choosing an extension $\rho$ of it by a crystalline lift $\chi_1$ of $\ochi_1$ such that 
\begin{equation}\label{scalar-case}
\xymatrix{
0\ar[r]&  H^1(G_{\Q_p},\chi_1)_{{\rm tf}}\ar[r]\ar[d]^{\iota_1}&  H^1(G_{\Q_p},\rho)_{{\rm tf}}\ar[r] \ar[d]^{\iota_3}& 
  H^1(G_{\Q_p},\rho_{2,m_2,m_3})_{{\rm tf}}\ar[d]^{\iota_2}  \ar@/_25pt/[l]^{\text{lift $\widetilde{B}_2$ to $\widetilde{B}_3$}} 
  \ar[r] & 0  \\
&H^1(G_{\Q_p},\ochi_1)\ar[r]& H^1(G_{\Q_p},\br_3) \ar[r]& H^1(G_{\Q_p},\br_2)&
}
\end{equation}
where the vertical arrows mean the reduction maps. 
Note that none of them is necessarily surjective. 
When either of $\ochi_2,\chi_3$ is trivial or $\ve$, by using the notion in 
Definition \ref{marking} and Proposition 9.12 in \cite{yam2}, one can lift 
$B_2$ by suitably choosing a crystalline lift $\br_{2,m_2,m_3}$ of $\br_2$ as above.  
Let $\widetilde{B}_2$ is such a lift of $\widetilde{B}_2$ to $H^1(G_{\Q_p},\rho)_{{\rm tf}}$. 
When $\ochi_1$ is the 
trivial character or $\ve$, by using the marking in Definition \ref{marking} again and  Proposition 9.12 of \cite{yam},  one can lift $B$ to $H^1(\Gp,\rho)$ as a crystalline extension. This is also a diagonalizable representation by \cite[Lemma 1.4.3-(1)]{BGGT}.  
In fact, lifts of $\iota_3(\widetilde{B}_3)$ and $B$ differ by 
an element in the line of $H^1(G_{\Q_p},\ochi_1)$ defined 
by the marking. Since $H^1(G_{\Q_p},\chi_1)_{{\rm tf}}$ is of rank one when 
$\chi_1$ is chosen suitably and its reduction is surjective onto that line.  
This makes sure the existence of a lift. 
According to the above procedure, we can choose $(m_2,m_3)$ for $\rho_{2,m_2,m_3}$ as 
follows:
$$(m_2,m_2-m_3)=
\left\{
\begin{array}{ll}
(a+\ds\max_{2\le i\le 3}\{r_{\ochi_i,b_{i}}\},b)+\delta^p_{b0} & \text{if $a>b$} \\
(a+p-1,a) & \text{if $a=b$} \\
(a+p-1+r_{\ochi_3,b_3},b) & \text{if $a<b$}. 
\end{array}\right.
$$
Note that when $a<b$, it holds that $r_{\ochi_3,b_3}=p-1$ if and only if $(a,b)=(0,p-2)$ and $b_3$ is 
tr\`es ramifi\'ee.  

According to the above Hodge-Tate numbers, we define $w(\br)=c$ and 
$$(k_1(\br),k_2(\br))=
\left\{
\begin{array}{ll}
(1,2)+(a+\ds\max_{1\le i\le 3}\{r_{\ochi_i,b_{i}}\},b)+\delta^p_{b0} & \text{if $a>b$} \\
(1,2)+(a+p-1,a) & \text{if $a=b$} \\
(1,2)+(a+p-1+r_{\ochi_3,b_3},b) & \text{if $a<b$}. 
\end{array}\right.
$$
\subsubsection{The case of the tr\`es ramifi\'ee or ramified extensions}
In this case, since $\ot_0$ is tr\`es ramifi\'ee or ramified,  
$b=1$ or 0 respectively. Further, $r_{\ochi^{-1}_3\ochi_2,\ot_0}=p-1$ by definition and 
we keep to use $r_{\ochi^{-1}_3\ochi_2,\ot_0}$ to clarify how $\ot_0$ affects 
to the classical Serre weights. 
As in the previous case one can lift $\br_2$ to a (potentially diagonalizable) crystalline lift $\rho_{2,m_2,m_3}$
of Hodge-Tate weights $\{m_1,m_2\}$ as follows:
$$(m_2,m_2-m_3)=
\left\{
\begin{array}{ll}
(a+\ds\max_{2\le i\le 3}\{r_{\ochi_i,b_{i}}\}+r_{\ochi^{-1}_3\ochi_2,\ot_0},b+
r_{\ochi^{-1}_3\ochi_2,\ot_0})+\delta^p_{b0} & \text{if $a>b$} \\
(a+p-1+r_{\ochi^{-1}_3\ochi_2,\ot_0},a+r_{\ochi^{-1}_3\ochi_2,\ot_0})+\delta^p_{a0} & \text{if $a=b$} \\
(a+p-1+r_{\ochi_3,b_3}+r_{\ochi^{-1}_3\ochi_2,\ot_0},b+r_{\ochi^{-1}_3\ochi_2,\ot_0}) & \text{if $a<b$}. 
\end{array}\right.
$$
According to the above Hodge-Tate numbers, we define $w(\br)=c$ and 
$$
(k_1(\br),k_2(\br))=
\left\{
\begin{array}{ll}
(1,2)+(a+\ds\max_{1\le i\le 3}\{r_{\ochi_i,b_{i}}\}+r_{\ochi^{-1}_3\ochi_2,\ot_0},b+
r_{\ochi^{-1}_3\ochi_2,\ot_0})+\delta^p_{b0} & \text{if $a>b$} \\
(1,2)+(a+p-1+r_{\ochi^{-1}_3\ochi_2,\ot_0},a+r_{\ochi^{-1}_3\ochi_2,\ot_0})+\delta^p_{a0} & \text{if $a=b$} \\
(1,2)+(a+p-1+r_{\ochi_3,b_3}+r_{\ochi^{-1}_3\ochi_2,\ot_0},b+r_{\ochi^{-1}_3\ochi_2,\ot_0}) & \text{if $a<b$}. 
\end{array}\right.
$$

\subsubsection{A uniform formula}
In the case of Borel ordinary, in general, we may define 
$w(\br)=c$ and 
$$
(k_1(\br),k_2(\br))=
\left\{
\begin{array}{ll}
(1,2)+(a+\ds\max_{1\le i\le 3}\{r_{\ochi_i,b_{i}}\}+r_{\ochi^{-1}_3\ochi_2,\ot_0},b+
r_{\ochi^{-1}_3\ochi_2,\ot_0})+\delta^p_{b0} & \text{if $a>b$} \\
(1,2)+(a+p-1+r_{\ochi^{-1}_3\ochi_2,\ot_0},a+r_{\ochi^{-1}_3\ochi_2,\ot_0})+\delta^p_{a0} & \text{if $a=b$} \\
(1,2)+(a+p-1+r_{\ochi_3,b_3}+r_{\ochi^{-1}_3\ochi_2,\ot_0},b+r_{\ochi^{-1}_3\ochi_2,\ot_0}) & \text{if $a<b$}. 
\end{array}\right.
$$

\subsection{Siegel ordinary case}\label{so} 
As explained in Section 9.4.2 of \cite{yam2}, 
the extension $\ot_2$ is related to $\ot_1$ under an isomorphism between 
local Galois cohomologies. Therefore, we may consider 
$$\br_3:=
\begin{pmatrix}
 \ochi & \ot_1 \\
0 & \br_1
\end{pmatrix},\ \ochi:=\op_2\ve^{a+b}.$$
Since $\ot_1$ is liftable to a potentially diagonalizable cyrtalline extension since $\br_1$ is also liftable to a cyrtalline, 
potentially diagonalizable 
representation by the argument in Section 9.3.2 of \cite{yam2} with the main result of \cite{GLiu}. Therefore, we have only to 
consider $\ot_3$ which is controlled by a line 
in $H^1(\Gp,\ochi)$ as in the case of Borel ordinary. 
Notice that $\omega^{b+pa}_2$ does not change 
when we replace $(a,b)$ with $(a,b)+m(p-1,p-1)$ for any $m\in \Z$.  
Then, in this case, we define 
$w(\br)=c$ and 
$$
(k_1(\br),k_2(\br))=
(1,2)+(a+r_{\ochi,\ot_3},b+r_{\ochi,\ot_3})+\delta^p_{b0}. 
$$

\subsection{Klingen ordinary case}\label{ko} 
Put $\br_3:=(\op_0\ve^{c})^{-1}{\rm ad}^0(\br_1)$. 
The extension class corresponding to $\br$, say $B$, belongs to $H^1(\Gp,\br_3)$. 
Since $p>2$, $${\rm ad}^0({\rm Ind}^{G_{\Q_p}}_{G_{\Q_{p^2}}}\omega^{b+ap}_2)=\overline{\delta}\oplus
{\rm Ind}^{G_{\Q_p}}_{G_{\Q_{p^2}}}\omega^{(p-1)(a-b)}_2$$ where 
$\overline{\delta}:G_{\Q_p}\lra \bF^\times_p$ is the unramified quadratic character. 
The second component is irreducible if and only if $(p+1)\nmid 2(a-b)$. 
If $(p+1)|2(a-b)$, put $m=\ds\frac{2(a-b)}{p+1}$. 
Since $0<2(a-b)<2(p-1)$, $m=1$. 
Then 
${\rm Ind}^{G_{\Q_p}}_{G_{\Q_{p^2}}}\omega^{(p-1)(a-b)}_2=\ve^{\frac{p-1}{2}}
\oplus\overline{\delta}\ve^{\frac{p-1}{2}}$. 
Now we apply the argument in Section \ref{so} for the former case (hence, when $(p+1)|2(a-b)$) and 
in Section \ref{trivial-ext} and Section 9.3.2 of \cite{yam2} with the main result of \cite{GLiu} for the latter case respectively. 

Put $\ochi=(\op_0\ve^{c})^{-1}\overline{\delta}$ if $(p+1)\nmid 2(a-b)$ and 
we denote by $\ot$ the extension class corresponding to $B$ under 
the projection $H^1(\Gp,\br_3)\lra H^1(\Gp,\ochi)$ defined by the decomposition of 
${\rm ad}^0(\br_1)$. 
In this case we define $(k_1(\br),k_2(\br),w(\br))$ by 
$$
\left\{
\begin{array}{ll}
(1,2,0)+(a+b-c+r_{\ochi,\ot},a-b,c-r_{\ochi,\ot}-a) & \text{if $2b>c-r_{\ochi,\ot}$} \\
(1,2,0)+(a+b+2(p-1)-c+r_{\ochi,\ot},a-b,c-r_{\ochi,\ot}-a-(p-1)) & \text{if $2b\le c-r_{\ochi,\ot}$}. 
\end{array}\right.
$$
When $(p+1)|2(a-b)$, hence $a-b=\frac{p+1}{2}$, put 
$$\ochi_1=(\op_0\ve^{c})^{-1}\overline{\delta},\ 
\ochi_2=(\op_0\ve^{c})^{-1}\ve^{\frac{p-1}{2}},\ \ochi_3=
(\op_0\ve^{c})^{-1}\overline{\delta}\ve^{\frac{p-1}{2}}.$$
For each $1\le i\le 3$, we denote by $\ot_i$ the extension class corresponding to $B$ under 
the projection $H^1(\Gp,\br_3)\lra H^1(\Gp,\ochi_i)$ defined by the decomposition of 
${\rm ad}^0(\br_1)$ as before. 
Put $$r_{\br_3,B}:=\ds\max_{1\le i\le 3}\{r_{\ochi_i,b_{i}}\}$$ for simplicity. 
Then we define $(k_1(\br),k_2(\br),w(\br))$ by 
$$
\left\{
\begin{array}{ll}
(1,2,0)+(a+b-c+r_{\br_3,B},a-b,c-r_{\br_3,B}-a) & \text{if $2b>c-r_{\br_3,B}$} \\
(1,2,0)+(a+b+2(p-1)-c+r_{\br_3,B},a-b,c-r_{\br_3,B}-a-(p-1)) & \text{if $2b\le c-r_{\br_3,B}$}. 
\end{array}\right.
$$

\subsection{Endoscopic case}
We may assume $a_1+b_1\ge a_2+b_2$. 
Since $\overline{\chi}_p^{a_1+b_1}|_{I_{\Q_p}}=\det(\br_1)|_{I_{\Q_p}}=\det(\br_2)|_{I_{\Q_p}}=\overline{\chi}_p^{a_2+b_2}|_{I_{\Q_p}}$, 
we have $a_1+b_1\equiv a_2+b_2$ mod $p-1$. Thus, $a_1+b_1=a_2+b_2$ or $a_1+b_1=a_2+b_2+p-1$ since $2\le a_i+b_i\le 2p-3$. 

Let us first consider the case $a_1+b_1=a_2+b_2$. 
Since $\br_1|_{I_{\Q_p}}\not\simeq \br_2|_{I_{\Q_p}}$, $a_1\neq  a_2$ and $b_1\neq b_2$.  
If $a_1>a_2$, then $b_1<b_2$. Thus, $a_1>a_2>b_2>b_1$. 
By the argument in \cite[Section 9.3.2]{yam2} one can construct a potentially diagonalizable, crystalline lift $\rho_i$ of $\br_i$ 
such that $HT(\rho_i)=(a_i,b_i)$. In this case, we define 
$$(k_1(\br),k_2(\br),w(\br))=(a_2-b_1+1,b_2-b_1+2,b_1).$$
The case when $a_1<a_2$ (and hence $b_1>b_2$) is similarly handled and 
the classical Serre weights are defined as follows: 
$$(k_1(\br),k_2(\br),w(\br))=(a_1-b_2+1,b_1-b_2+2, b_2).$$

Now, we consider the case $a_1+b_1=a_2+b_2+p-1>a_2+b_2$.  
In this case, we can easily see that $a_1,b_1,a_2,b_2$ are distinct each other since 
$0\le b_i<a_i<p$. Let $A_1>A_2>B_2>B_1$ such that $\{A_1,A_2,B_1,B_2\}=\{a_1,b_1,a_2,b_2\}$. 
Then we define 
$$(k_1(\br),k_2(\br),w(\br))=(A_2-B_1+1,B_2-B_1+2,B_1).$$

\subsection{Irreducible case}
We require a development of the integral $p$-adic Hodge theory to handle this case 
for any $p$.   
However, by \cite[p.176, Theorem 6.4.4-4. and p.177, Remark 6.4.5]{EG},  
there exists a $(k_1,k_2,w)\in {\rm SW}(\br)$ such that 
$k_1+k_2-3\le 4p-1$.  
Thus, we can define $(k_1(\br),k_2(\br),w(\br))$ to be a minimum element of  
${\rm SW}(\br)$ with respect to the lexicographic order in the first and second coordinates 
(so that we first compare the second coordinate). 
 
\section{Classical Serre weights for $p=2$}\label{CSW2}
In this case $\ve=\textbf{1}$ and $H^1(G_{\Q_2},\F)\simeq 
{\rm Hom}(\Q^\times_2/(\Q^\times_2)^2,\F)$ is generated by three classes.
By using these classes, we make markings as in Definition \ref{marking} 
to lift extension classes. 
Then the same argument mostly works and the classical 
Serre weights similarly are defined except for 
the case of Klingen. In this case, in the notation of Section \ref{ko}, 
$$0\lra (\br_1)^{(2)} \lra \br_3\lra \F(\op^2_1\op^{-1}_0)\lra 0$$
where $(\br_1)^{(2)}$ stands for the Frobenius twist defined by 
$2$-th power map. Even though this sequence is non-split,  
applying Proposition 9.15 of \cite{yam2}, $\br_3$ is liftable to 
a (potentially diagonalizable) crystalline lift. 
Put $\ochi=\op^2_1\op^{-1}_0$. 
We denote by $\ot$ the extension class corresponding to $B$ under 
the projection $H^1(\Gp,\br_3)\lra H^1(\Gp,\ochi)$ defined by the above sequence. 
Notice that ${\rm Ind}^{G_{\Q_2}}_{G_{\Q_{2^2}}}\omega_2\simeq 
{\rm Ind}^{G_{\Q_2}}_{G_{\Q_{2^2}}}\omega^{2}_2$. 
Therefore, we may assume $(a,b)=(1,0)$. 
In this case we define $(k_1(\br),k_2(\br),w(\br))$ as in the former case of 
Section \ref{ko} by substituting $(a,b)=(1,0)$.  

\section{A proof and a supplemental result}\label{proof-main2}\label{proof}
\subsection{A proof of Theorem \ref{main2}}
By the construction in Section \ref{CSW1}, one can take a 
potentially diagonalizable, crystalline lift  of $\br|_{G_{\Q_p}}$ with regular Hodge-Tate weights 
$$\{k_1(\br)+k_2(\br)-3+w(\br),k_1(\br)-1+w(\br),k_2(\br)-2+w(\br)
,w(\br)\}$$
satisfying $k_1(\br)\ge k_2(\br)\ge 3$. 
Applying Theorem 1.5 of \cite{yam2}, we have the claim. 
The remaining part for the theta cycle 
follows from the contents in Section \ref{the-theta-cycle}. 

\subsection{A potentially diagonalizable lift of the minimal Hodge-Tate weights}
If $\br|_{\Gp}$ is irreducible, then by Lemma 2.1.12 of \cite{GHTS}, 
it has, up to the twist by a power of $p$-adic cyclotomic character, a potentially diagonalizable lift of Hodge-Tate weight $\{0,1,2,3\}$. 
Combining it with Theorem 1.5 of \cite{yam2} we have the following:
\begin{thm}\label{main3}
Let $\br:G_\Q\lra \GSp_4(\bF_p)$ be a mod $p$ Galois representation satisfying 
\begin{itemize}
\item $p\ge 3$; 
\item $\br|_{G_{\Q(\zeta_p)}}$ is irreducible and ${\rm Im}(\br)$ is adequate; 
\item $\br|_{\Gp}$ is irreducible; 
\item $\br\sim \br_{f,p}$ for some cuspidal Hecke eigenform $f$ in $M_{\uk'}(N)$ 
with some weight $\uk'=(k'_1,k'_2),\ k'_1\ge k'_2\ge 3$.  
\end{itemize}
Then there exist a cuspidal Hecke eigen form $g$ unramified outside $pN$ 
of weight $(3,3)$ and an integer $w(\br)$ 
such that $\br\simeq \ochi^{w(\br)}_p\otimes\br_{g,p}$ and 
$\rho_{g,p}|_{\Gp}$ is potentially diagonalizable. 
\end{thm}
This would suggest to study the mod $p$ reductions of potentially crystalline representations of the minimal regular Hodge-Tate weights $\{0,1,2,3\}$ as it is done for $GL_3$ in 
\cite{Park},\cite{LLLM}.

\end{document}